\documentclass[10pt,twoside]{article}
\usepackage{amsmath,amssymb,amsthm,color}
\usepackage[T1]{fontenc}
\usepackage[latin1]{inputenc} 
\usepackage{a4wide}
\usepackage{amsmath,amsthm,amsxtra}
\usepackage{xparse}
\usepackage{tabularx}
\usepackage{multicol}
\usepackage{color}
\usepackage{yfonts}
\usepackage{float}
\usepackage{graphicx}
\usepackage{subcaption}
\usepackage{pdfsync}
\usepackage{graphicx}

\usepackage[colorlinks=true]{hyperref}
\hypersetup{linkcolor=red,citecolor=blue,filecolor=dullmagenta,urlcolor=blue}

\newtheorem{theorem}{Theorem}[section]
\newtheorem{corollary}[theorem]{Corollary}
\newtheorem{lemma}[theorem]{Lemma}
\newtheorem{proposition}[theorem]{Proposition}
\newtheorem{claim}[theorem]{Claim}

\theoremstyle{remark}

\newtheorem{remark}{Remark}[section]

\newcommand{\deri}{\frac{d}{dt}}
\newcommand{\intR}{\int_\mathbb{R}}
\newcommand{\intp}{\int_0^{\infty}}
\newcommand{\im}{\text{Im}}
\newcommand{\re}{\text{Re}}
\newcommand{\R}{\mathbb{R}}

 \numberwithin{equation}{section}

\DeclareMathOperator{\sech}{sech}

\begin{document}

\title{On the decay problem for the Zakharov and Klein-Gordon Zakharov systems
in one dimension }
\author{Mar\'ia E. Mart\'inez\\ Departamento de Ingener\'ia Matem\'atica DIM \\
FCFM Universidad de Chile\\
{\tt maria.martinez.m@uchile.cl}\footnote{Partially funded by Conicyt Beca
Doctorado Nacional no. 21192076, ECOS-Sud C18E06, CMM Conicyt PIA AFB170001 and
Fondecyt Regular
1191412. Part of this work was done while the author was visiting
Universit\'e Paris-Saclay (Sud) and
 Evolutionary Partial Differential Equation group (Universidad de Granada),
 whose support is greatly acknowledged.}}
\date{\today}
\maketitle

\begin{abstract}
  We are interested in the long time asymptotic behavoir of solutions to the
  scalar Zakharov system
  \begin{equation*}
  \begin{array}{ll}
  i u_{t} + \Delta u = nu,\\
  n_{tt} - \Delta n= \Delta |u|^2
  \end{array}
  \end{equation*}
  and the Klein-Gordon Zakharov system
  \begin{equation*}
  \begin{array}{ll}
  u_{tt} - \Delta u + u = - nu,\\
  n_{tt} - \Delta n= \Delta |u|^2
  \end{array}
  \end{equation*}
  in one dimension of space.
  For these two systems, we give two results proving decay of solutions for
  initial data in the energy
  space. The first result deals with decay over compact intervals asuming
  smallness and
  parity conditions ($u$ odd). The second result proves decay in far field
  regions along curves for solutions whose growth can be dominated by an
  increasing $C^1$ function. No smallness condition is needed to prove this last
   result for the Zakharov system.
  We argue relying on the use of suitable virial
  identities appropiate for the equations and follow the technics of
  \cite{KMM1, M} and \cite{MPS}.
 \end{abstract}

\section{Introduction}

In this work, we are concerned with the one dimensional Zakharov system
	\begin{equation}\label{Zakharov}
	\begin{array}{ll}
 	i u_{t} + \Delta u = nu,& \quad (t,x) \in \R\times \R,\\
	\alpha^{-2} n_{tt} - \Delta n= \Delta |u|^2,& \quad (t,x) \in \R\times \R,
  \end{array}
	\end{equation}
with initial data
	\[
	u(t=0, x)=u_0(x), \quad n(t=0, x)=n_0(x), \quad n_t(t=0, x)=n_1(x).
	\]
where $u(t,x):\R\times \R\to \mathbb{C}$, $n(t,x):\R\times \R\to \R$ and
$\alpha>0$.
\par We are also interested the Klein-Gordon Zakharov system in one dimension
  \begin{equation}\label{KGZ}
  \begin{array}{ll}
  c^{-2}u_{tt} - \Delta u + c^2u = - nu,& \quad (t,x) \in \R\times \R,\\
  \alpha^{-2} n_{tt} - \Delta n= \Delta |u|^2& \quad (t,x) \in \R\times \R,
  \end{array}
  \end{equation}
  with initial data
  \[
  \begin{aligned}
  u(t=0, x)=u_0(x), \quad  u_t(t=0, x)=u_1(x) \\
  n(t=0, x)=n_0(x), \quad  n_t(t=0, x)=n_1(x).
  \end{aligned}
  \]
where $u(t,x):\R\times \R\to \mathbb{R}$, $n(t,x):\R\times \R\to \R$, $\alpha>0,
 \ c>0$.

\medskip

\par The Zakharov systems are simplified models for the description of
long-wavelenght small-amplitude Langmuir oscillations in a ionized plasma
\cite{Z}. Langmuir waves are rapid oscillations of the electron density;
electrons and ions oscillate out of phase.  Zakharov equations model the
nonlinear interactions between the mean mode of the ionic fluctuations of
density in the plasma $n$ and the changing amplitude of electric field $u$,
which varies slowly compared to the unperturbed plasma frequency. The constant
$\alpha$ is the ion sound speed and $c$ is the plasma frequency.
\par In the \emph{subsonic limit} ($\alpha \to \infty$), in which density
perturbations are changing
slowly, the term $n_{tt}$ of the wave equation in \eqref{Zakharov} is
negligible. This would imply that the Langmuir waves follow the cubic NLS
equation
  \[
  i u_t +\Delta u +|u|^2u=0.
  \]
  If one considers $\tilde u=e^{ic^2t}$ in \eqref{KGZ}, then it follows that
  \begin{equation*}
 \begin{aligned}
 &c^{-2}\tilde u_{tt} - 2i\tilde u_t - \Delta \tilde u = - n\tilde u,\\
 &\alpha^{-2} n_{tt} - \Delta n= \Delta |\tilde u|^2.
  \end{aligned}
 \end{equation*}
Thus, formally, in the \emph{high-frequency limit} (that is, taking
$c \to \infty$) of the Klein-Gordon-Zakharov \eqref{KGZ}, the Zakharov system
\eqref{Zakharov} is recovered.
\par These high-frequency and subsonic limits were extensively studied in
\cite{MN2}-\cite{MN1}. See, also, \cite{SS2, R, Ber} for more details on the
physical derivation.

\medskip

The Zakharov system \eqref{Zakharov}
preserves the mass $\|u(t)\|_{L^2(\R)}=\|u(0)\|_{L^2(\R)}$ and the energy
	\begin{equation*}
	\begin{aligned}
	H_{S}(t):= \intR |\nabla u(t,x)|^2
  + \frac{1}{2}\Big(|n(t,x)|^2 +\frac{1}{\alpha^2}|D^{-1}n_t(t,x)|^2\Big)
  + n(t,x)|u(t,x)|^2dx,
	\end{aligned}
	\end{equation*}
where $D=\sqrt{-\Delta}$. The Klein-Gordon-Zakharov system \eqref{KGZ} preserves
 the following energy, as well:
\begin{equation*}
\begin{aligned}
H_{KG}(t)=& \intR c^2|u(t,x)|^2 + |\nabla u(t,x)|^2 + \frac{1}{c^2}|u_t(t,x)|^2
+ \frac{1}{2}|n(t,x)|^2 \\
&+\frac{1}{2}||\alpha D|^{-1}n_t(t,x)|^2 + n(t,x)|u(t,x)|^2dx=H_{KG}(0).
\end{aligned}
\end{equation*}

\medskip

\par The system \eqref{Zakharov} in one dimension is globally well-posed for
initial data in $H^1(\R)\times L^2(\R)\times \hat H^{-1}(\R)$, where
	\[
	 w \in \hat H^s \text{ if there exists } v:\R^d\to \R^d \text{ such that }
   w=\nabla\cdot v \text{ and } \|w\|_{\hat H^s}=\|v\|_{H^{s+1}}.
	 \]
The first aproach in
this regard was presented by Sulem and Sulem in \cite{SS}, where they stated
local well-posedness of \eqref{Zakharov} for dimensions $d=1,2,3$ and initial
data
  \[
  (u_0, n_0, n_1)\in H^m(\R^d) \times H^{m-1}(\R^d)\times \left(H^{m-2}\cap
  \hat H^{-1}\right)(\R^d), \quad m\ge 3.\]
Using the Brezis-Gallou\"et inequality
\begin{equation}\label{BrezisGallouet}
\|u\|_{L^{\infty}}\lesssim 1+\|u\|_{H^1}(\ln(1+\|\Delta u\|_{L^2})),
\end{equation}
valid if $u \in H^2(\R^2)$,
Added and Added in \cite{AA} improved \cite{SS} to global well-posedness for
small initial data in the 2-dimensional case. The local well posedness result
($d=1,2,3$) was refined by Ozawa and Tsutsumi \cite{OT}, for data
$(u_0, n_0, n_1) \in H^2\times H^1\times L^2$,
and by Colliander \cite {C1} for
$(u_0, n_0, n_1) \in H^1\times L^2\times\hat H^{-1}$.
In fact, in \cite{C1}, using apriori estimates on the $H^1$-norm of $u$,
Colliander shows global well-posedness for small data in the one-dimensional
case. Finally, on \cite{P}, Pecher proves that the Zakharov system is globally
well posed for rough data
$(u_0, n_0, n_1) \in H^s\times L^2\times \hat H^{l-1}$, $1>s>9/10$ for
dimension $d=1$ without any smallness condition. More results on local and
global well-posedness for other dimensions and more general nonlinearities are
stated in \cite{BC, C2, GTV, MN6}, and on the torus in \cite{B}.

\medskip

\par Regarding well-posedness for system \eqref{KGZ}, the first result was
presented in \cite{OTT1}, where the authors followed a method based on the
theory of normal forms to prove that \eqref{KGZ} in dimensions $d=1,2,3$ with
$c= \alpha=1$, admits a unique global solutions for small initial data with
rather restrictive regularity conditions.
Also, they give a completness result for the $3$-dimensional case, as they
show the existence of global solutions that tend to behave assymptotically
(when $t\to \infty$) as the free solutions.
\par Using Sobolev invariant spaces, Tsutaya \cite{Tsutaya}, improved the
regularity conditions of the global existence result in \cite{OTT1} for
dimension $d=3$. The low-frequency case ($0<\alpha<1$) in three dimensions was
adressed in \cite{OTT2}, where the authors rely on the different propagation
speed (normalized $c=1$ in the Klein gordon equation, while assumed $0<\alpha<1$
 in the wave equation) to prove local and then global well-posedness for small
 initial data in the energy space:
  \[
  (u_0, u_1, n_0, n_1) \in H^1\times L^2\times L^2\times \hat H^{-1}.
  \]
Following the idea stated in \cite{OTT2}, Otha and Todorova \cite{OthaT}
extended the well-posedness result for all $\alpha>0$ and $d<3$. The
high-frequency, subsonic case in dimension $d=3$ was later treated by Masmoudi
and Nakanishi in \cite{MN5}-\cite{MN6}, where they presented local
well-posedness in the energy space under the assumption $\alpha<c$.

\medskip

In this paper, we are interested in the decay of solutions to systems
\eqref{Zakharov} and \eqref{KGZ}.

\medskip

\par It is known that for dimension $d=2,3$, there exist solutions to
\eqref{Zakharov}
 that decay to zero in the energy space $H^1\times L^2 \times \hat H^{-1}$.
 Indeed, Ozawa and Tsutsumi \cite{OT}, Ginibre and Velo \cite{GV}, and Shimomura
  \cite{S}, proved existence and uniqueness of asymptotically free solutions of
  \eqref{Zakharov} by solving the system with final data given at $t\to \infty$,
   instead of the initial value problem; that is, for $u_+$ and $n_+$ free
   solutions of the Schr\"odinger and wave equations respectively,
	\[
	\|u(t)-u_+(t)\|_{H^1} + \|\nabla n(t)-\nabla n_+\|_{L^2}  + \|\partial_t n(t)
  -\partial_t n_+(t)\|_{L^2}\to 0, \text{ as }t\to \infty.
	\]
Completness results were also obtained for the 3-dimensional
Klein-Gordon-Zakharov \eqref{KGZ}. In fact, Ozawa, Tsutaya and Tsutsumi
\cite{OTT1} proved the existence of global solutions that behave asymptotically
as free solutions in space
  \[
    \sum_{j=0,1} \|\partial_t^j(u(t)-u_+(t))\|_{H^{52-j}}
    +\sum_{j=0,1}\|\partial_t^j(n(t)-n_+(t))\|_{H^{51-j}}\to 0 \quad \text{as}
    \quad t\to \infty.
  \]
\par
In \cite{GN}, Guo and Nakanishi prove that radially symmetric solutions for the
Zakharov system in $d=3$ with small energy do scatter. By using generalized
Strichartz estimates for the Schr\"odinger equation, Guo, Lee, Nakanishi and
Wang in \cite{GLNW} were able to improve \cite{GN} by showing scattering of
small solutions without the radial assumption.

Following the idea in \cite{GN}, Guo, Nakanishi and Wang \cite{GNW1} proved
scattering in the energy space for radially symetric solutions with small energy
 for the system \eqref{KGZ} in three dimensions, as well. In \cite{GNW2}, they
 continue the study of global dynamics of radial solutions in three dimensions
 and find a dichotomy between scattering and blow-up. More specifically, relying
  on virial identities, they show that if the initial data is radially symetric
  and its energy is below the energy of the ground state then the solution to
  \eqref{KGZ} can either (for both $i=1,2$):
\begin{itemize}
  \item scatter when $J_i(u_0)\ge0$, or
  \item blow up in finite time when $J_i(u_0)<0$,
\end{itemize}
where $J_i$ are scaling derivative of the static Klein-Gordon energy:
  \begin{gather*}
    J_0(v)=\int |u|^2+|\nabla u|^2 -|u|^4dx \quad \text{and} \quad
    J_2(v)=\int |\nabla u|^2 -\frac{3}{4}|u|^4dx.
  \end{gather*}
\par The behavior of radially symmetric solutions of \eqref{Zakharov} was also
studied in \cite{M1}. Using virial identities, Merle in \cite{M1} showed blow
up at either finite or infinity time for radially symmetric solutions to
\eqref{Zakharov} that satisfy $E_s<0$ for $d= 2,3$. In \cite{M2}, Merle
improved the results in \cite{M1} by presenting lower estimates on the blow-up
of the Zakharov system in the 2-dimensional case.
\par Notice that all positive decay/scattering results above mentioned do not
deal with the case $d=1$.
\par

\medskip

\par From now on, we consider the one-dimensional case. In the following
subsections, we introduced a reduction of order for the Zakharov and
Klein-Gordon-Zakharov systems and present the main results of this work.

\subsection{Main results for Zakharov system}

In order to simplify the computations, from now on we consider system
\eqref{Zakharov} with $\alpha=1$, although the analysis still works for
$\alpha\ne 1$. With the purpose of reducing \eqref{Zakharov} into a first order
system, we
introduce the real function $v$ such that:
  \begin{equation}\label{ZSb}
  \begin{aligned}
    &i u_{t} + u_{xx} = nu, \\
    &n_{t} + v_x= 0, \\
    &v_{t} + \left( n +|u|^2\right)_x = 0,
  \end{aligned}
  \end{equation}
and
  \[
  u(t=0, x)=u_0(x), \ n(t=0, x)=n_0(x), \ v(t=0, x)=v_0.
  \]
Such suposition is possible because we study the Zakharov system
\eqref{Zakharov} in the Hamiltonian case, meaning that we assume that there
exists $v_0 \in L^{2}(\R)$ such that $-\nabla \cdot v_0 = n_t(0)$; property that
is preserved by the flow. This way, to consider
$(u, n, n_t) \in H^1(\R)\times L^2(\R)\times \dot H^{-1}(\R)$ solution of
\eqref{Zakharov} is equivalent to study
$(u, n, v) \in H^1(\R)\times L^2(\R)\times L^2(\R)$ solution to \eqref{ZSb}.

\medskip

\par
 The system \eqref{ZSb} preserves:
\begin{itemize}
	\item Mass: \begin{equation}\label{CQ0}
	\begin{aligned}
	M_s(t):= \intR | u(t,x)|^2 dx=M_s(0),
	\end{aligned}
	\end{equation}
	\item Energy:
	\begin{equation}\label{CQ1}
	\begin{aligned}
	E_{s}(t):= \intR | u_x(t,x)|^2 + \frac{1}{2}\Big(|n(t,x)|^2
  +|v(t,x)|^2\Big) + n(t,x)|u(t,x)|^2dx
	=E_s(0),
	\end{aligned}
	\end{equation}
	\item Momentum:
	\begin{equation}\label{CQ2}
	\begin{aligned}
	P_s(t):=\im \intR u(t,x)\overline{u}_x(t,x) dx
  - \intR v(t,x)n(t,x) \text{ } dx = P_s(0).
	\end{aligned}
	\end{equation}
\end{itemize}

\medskip

In the present work, we show decay for solutions of \eqref{ZSb} in one dimension
 in two different ways. On one hand, we prove decay on any compact interval for
 solutions to \eqref{ZSb} under parity assumptions ($u$ odd). On the other hand,
  we are able to show decay, without any oddness condition but with sufficient
  regularity ($\|u(t)\|_{H^2} \in L^\infty$), in regions along curves outside
  the ``light cone''.

\begin{theorem}\label{theorem1}
	Assume $E_s<\infty$. Let
  $(u, n, v) \in C\left(\R^+, H^1(\R)\times L^2(\R)\times L^2(\R)\right)$
  be a solution of \eqref{ZSb} such that $u$ is odd and satisfies, for some
  $\varepsilon>0$ small,
		{\color{black}
		\begin{equation}\label{hyp}
    \sup_{t\ge 0} \|u(t)\|_{H^1(\R)}<\varepsilon.
		\end{equation}}
	Then, for every compact interval $I\subset \R$,
		\begin{equation}\label{thesis}
		\lim_{t \to \infty} \|u(t)\|_{L^\infty(I)}+\|u(t)\|_{L^2(I)}
    +\|n(t)\|_{L^2(I)}+\|v(t)\|_{L^2(I)} = 0.
		\end{equation}
\end{theorem}

\begin{remark}
	Asking for $u$ to be odd implies necessarily for $n$ to be even. This property
   is preserved by the flow.
\end{remark}
\begin{remark}
  The fact that $u$ is odd allows to rule out solitary waves.
  The first result regarding solitary waves was stated by Wu in \cite{W}, where
  he proves existence and orbital stability of solutions
    \begin{equation}\label{ExactTravelingWaves1}
    u(t,x)=e^{-i\omega t}e^{iq(x-ct)}u_{\omega, c}(x-ct),
    \end{equation}
  and
    \begin{equation}\label{ExactTravelingWaves2}
    n(t,x)=n_{\omega, c} (x-ct),
    \end{equation}
  for
    \[
    u_{\omega, c}(x)=\sqrt{\frac{(4\omega+c^2)(1-c^2)}{2}}
    \sech\left( \frac{\sqrt{4\omega+c^2}}{2} x\right),
    \]
    \[
    n_{\omega, c}(x)=\left(2\omega+\frac{c^2}{2}\right)
    \sech^2\left(\frac{\sqrt{4\omega+c^2}}{2}x\right), \quad q=\frac{c}{2},
    \]
  satisfying
    \[
    4\omega +c^2\ge 0 \ \text{ and } \ 1-c^2>0.
    \]
  \par Angulo and Banquet \cite{AB} studied existence of periodic travelling
  wave forms such as \eqref{ExactTravelingWaves1}-\eqref{ExactTravelingWaves2},
  in this case for $u_{\omega,c}$ and $n_{\omega, c}$ being periodic functions,
  and prove their orbital stability as well. See also \cite{FGG, O, H1, H2, ZSX}
   for other results on solitary waves for generalized Zakharov systems.
\end{remark}

\begin{remark}
  We do not prove decay in the energy space
  $H^1\times L^2\times \dot H^{-1}(\R)$.
  This is because uncontrolled $H^2$-terms emerge when considering semi-norm
  $\dot H^1$ for the solution $u$ of the Schr\"odinger equation. We show
  $L^\infty$ decay instead.
\end{remark}

\begin{remark}
  The result also holds for a generalized Zakharov system when adding a
  potential term $|u|^pu$ in the Schr\"odinger equation. See \cite{M} for the
  details on how to treat the new non-linear term.
\end{remark}

\begin{remark}
  In \cite{JCS1, JCS2}, the authors study the asymptotic behavior of the
  Zakharov-Rubenchik system. They prove that solutions blow up if the energy is
  negative and give inestability results for the solitary wave in the case
  $d=3$. Such system is of special interest since in the supersonic limit it
  has been proven that it converges to \eqref{Zakharov}.
\end{remark}

\medskip

The proof of Theorem \ref{theorem1} is based on the use of suitable virial
identities. The argument follows from \cite{KMM1,KMM2}, where the authors deal
with the Klein Gordon case. The idea is to argue as in \cite{M}, where a
functional adapted to the momentum for the nonlinear Sch\"odinger equation was
considered. Unfortunately, the identity used for the NLS equation, which allows
us to conclude, it is not appropiate in this case. Instead, as in \cite{M1, M2},
 we need to work with a virial identity that comes from the quantity
  \begin{equation*}
    \begin{aligned}
      P(t):=\im \intR u(t,x)\overline{u}_x(t,x) dx
      - \intR v(t,x)n(t,x) \text{ } dx.
    \end{aligned}
  \end{equation*}
Such virial has an uncontrolled term that we manage by adding the condition
\eqref{hyp}.

\medskip

Our second result deals with decay in far field regions along curves.

\begin{theorem}\label{theorem2}
	Assume $E_s<\infty$ and $M_s<\infty$. Let $(u, n, v)$ satisfy \eqref{ZSb}.
	\begin{enumerate}
		\item If
    $(u, n, v) \in C\left(\R^+, H^1(\R)\times L^2(\R)\times L^2(\R)\right)$,
    then, for any $\mu \in C^1(\R)$ satisfying
    \\$\mu(t)\gtrsim t\log(t)^{1+\delta}$, $\delta>0$,
			\begin{equation}\label{tesis1theorem2}
			\lim_{t \to \infty} \|u(t)\|_{L^2(|x|\sim \mu(t))}=0.
			\end{equation}
		\item If
    $(u, n, v) \in C\left(\R^+, H^2(\R)\times L^2(\R)\times L^2(\R)\right)$
    and there exists $f(t)\in C^1(\R)$ a
    non-decreasing function such that
			\begin{equation}\label{hip2theorem2}
			\|u(t)\|_{H^2(\R)}\lesssim f(t),
			\end{equation}
		then, for any $\mu \in C^1(\R)$ satisfying
    $\mu(t)\gtrsim t\log(t)^{1+\delta}f(t)$, $\delta>0$,
			\begin{equation}\label{tesis2theorem2}
			\lim_{t \to \infty} \|u(t)\|_{H^1(|x|\sim \mu(t))}
      +\|n(t)\|_{L^2(|x|\sim \mu(t))}+\|v(t)\|_{L^2(|x|\sim \mu(t))} = 0.
	\end{equation}
	\end{enumerate}
\end{theorem}

A direct consequence of the proof of Theorem \ref{theorem2} is the following
result for the NLS equation:

\begin{corollary}
 Let $u(t) \in H^1(\R)$ be a solution of the non-linear Schr\"odinger equation
  \[
    iu_t+u_{xx}\pm|u|^{p-1}u=0,
  \]
where $1<p<5$, with initial data $u(t=0, x)=u_0$ satisfying
$\|u(t=0)\|_{H^1(\R)}<\infty$. Then,
  \begin{equation*}
    \lim_{t \to \infty} \|u(t)\|_{L^2(|x|\sim \mu(t))}=0.
  \end{equation*}
\end{corollary}

The proof of Theorem \ref{theorem2} follows an argument recently introduced by
Mu\~noz, Ponce and Saut in \cite{MPS}, where they deal with the long time
behavoir of intermidiate long wave equation. This method proves to be
independent of the integrability of the equation and does not need size
restriction. However, when dealing with the Zakharov system, because of the
presence of uncontrolled $H^2$-terms in the dynamics of the $H^1$-norm of $u$,
we need the additional condition \eqref{hip2theorem2}. Note that such condition
allows as to obtain decay of the $\|\cdot\|_{H^1}$-norm, which was not present
in results established in \cite{MPS}.

\subsection{Main results for Klein-Gordon-Zakharov system}

We will consider system \eqref{KGZ} with $\alpha=c=1$, although the computations
still hold for different values of $\alpha, c \in \R$. As we did for the
Zakharov system \eqref{Zakharov}, we reduce \eqref{KGZ} by introducing a real
function $v$ satisfiyng $-\nabla \cdot v=n_t$ for all $t\ge 0$. That is, we get
a new first order system,
  \begin{equation}\label{KGZSv}
    \begin{aligned}
      &u_{tt} - u_{xx} + u = -nu,\\
      & n_{t} +  v_x= 0,\\
      &v_{t} + \left( n +|u|^2\right)_x= 0,
    \end{aligned}
  \end{equation}
and
  \[
  \begin{aligned}
    u(t=0, x)=u_0(x), \quad  u_t(t=0, x)=u_1(x) \\
    n(t=0, x)=n_0(x), \quad  v(t=0, x)=v_0(x).
  \end{aligned}
  \]
The system \eqref{KGZSv} preserves:
\begin{itemize}
	\item Energy:
	\begin{equation}\label{KGCQ1}
	\begin{aligned}
	E_{KG}(t):=& \intR |u(t,x)|^2 + |u_x(t,x)|^2 + |u_t(t,x)|^2 \\
  &+ \frac{1}{2}\Big(|n(t,x)|^2
  +|v(t,x)|^2\Big) + n(t,x)|u(t,x)|^2dx
	=E_{KG}(0),
	\end{aligned}
	\end{equation}
	\item Momentum:
	\begin{equation}\label{KGCQ2}
	\begin{aligned}
	P_{KG}(t):= \intR u_t(t,x){u}_x(t,x) dx
  - \frac12\intR v(t,x)n(t,x) \text{ } dx = P_{KG}(0).
	\end{aligned}
	\end{equation}
\end{itemize}

As we did for \eqref{ZSb}, we prove decay of solutions to \eqref{KGZSv} in two
different ways: over compact intervals of time and over far field regions along
curves. Our result for compact intervals is the following:

\begin{theorem}\label{theorem3}
	Assume $E_{KG}<\infty$. Let
  $(u, u_t, n, v) \in C\left(\R^+, H^1(\R)\times L^2(\R)
  \times L^2(\R)\times L^2(\R)\right)$
  be a solution of \eqref{ZSb} such that $u$ is odd and satisfies
  \begin{equation}\label{KGhyp}
  \sup_{t\ge 0} \|u(t)\|_{H^1(\R)}\le \varepsilon \quad \text{and} \quad
  \sup_{t\ge 0} \|u_t(t)\|_{L^2(\R)}\le C
\end{equation}
for some $C>0$ and $\varepsilon>0$ small.
	Then, for every compact interval $I\subset \R$,
		\begin{equation}\label{thesisKG}
		\lim_{t \to \infty} \|u(t)\|_{H^1(I)}
    +\|u_t(t)\|_{L^2(I)}
    +\|n(t)\|_{L^2(I)}+\|v(t)\|_{L^2(I)} = 0.
		\end{equation}
\end{theorem}

\begin{remark}
  The oddness condition rules out solitary waves. Indeed, solitary waves of
  \eqref{KGZSv} exist and they are orbitally stable. They were first introduced
  by Chen in \cite{C}, where he stated that solitons of the form
    \begin{gather*}
      u(t,x)=e^{-i\omega t} e^{iq(x-ct)}u_{\omega, c}(x-ct),\\
      n(t,x)=n_{\omega, c}(x-ct),
    \end{gather*}
  with
    \begin{gather*}
      u_{\omega, c}(x)=\sqrt{2(1-c^2-\omega^2)}
      \sech\left(\frac{\sqrt{1-c^2-\omega^2}}{1-c^2}x\right),\\
      n_{\omega, c}(x)=-2\frac{(1-c^2-\omega^2)}{1-c^2}
      \sech^2\left(\frac{\sqrt{1-c^2-\omega^2}}{1-c^2}x\right),
      \quad q=\frac{\omega c}{1-c^2}.
    \end{gather*}
  exists when the real constants $\omega$ and $c$ satisfy $1-c^2-\omega^2>0$.
  There also exist solitary waves
  $(u_{\omega, c}, {(u_t)}_{\omega, c}, n_{\omega, c}, v_{\omega, c})$ of
  \eqref{KGZSv} of the form
    \begin{gather*}
      u_{\omega, c}(x)=\sqrt{2(1-c^2-\omega^2)}
      \sech\left(\frac{\sqrt{1-c^2-\omega^2}}{1-c^2}x\right)
      e^{i\frac{\omega c}{1-c^2}x},\quad
      {(u_t)}_{\omega, c}=\left(i\omega
      +c\frac{\partial }{\partial x}\right)u_{\omega, c}(x),\\
      n_{\omega, c}(x)=-2\frac{(1-c^2-\omega^2)}{1-c^2}
      \sech^2\left(\frac{\sqrt{1-c^2-\omega^2}}{1-c^2}x\right), \\
      v_{\omega, c}(x)=2c\frac{(1-c^2-\omega^2)}{1-c^2}
      \sech^2\left(\frac{\sqrt{1-c^2-\omega^2}}{1-c^2}x\right),
    \end{gather*}
  with $1-2c^2-2\omega^2<0$ and are orbitally stable \cite{C}.
\end{remark}

\begin{remark}
  The result holds when cosidering cubic nonlinear KGZ system, that is, when
  adding an additional term $|u|^2u$ in the Klein Gordon equation. We do not
  adress this case in the proof, but it follows naturally from the analysis of
  the non-linear term in \cite{KMM1}.
\end{remark}

\medskip

The proof of this results follows more closely the idea in \cite{KMM1}. Indeed,
we construct a virial identity that comes from the momentum:
  \[
    P_{KG}=\intR u_t(t,x)u_x(t,x)dx - \frac12 \intR v(t,x)n(t,x)dx.
  \]
But, since there are uncontrolled terms involving $u$ and $u_t$ in the identity
from the potential, we need to consider $u_t$ uniformly bounded.
Notice that we obtain now decay in the whole energy norm (that is, even for the
$H^1$-norm).

\medskip

\par The last theorem is devoted to the decay of the solutions to \eqref{KGZSv}
in regions along curves outside the light cone:

\begin{theorem}\label{theorem4}
	Assume $E_{KG}<0$. If
  $(u, u_t, n, v) \in C\left(\R^+, H^1(\R)\times L^2(\R)
  \times L^2(\R)\times L^2(\R)\right)$ is a solution to \eqref{KGZSv}
  such that
      \begin{equation}\label{hip1theorem4}
        \sup_{t\ge 0}\|u(t)\|_{H^1(\R)}\le \varepsilon
      \end{equation}
    for some $0<\varepsilon\le1$,
		then, for any $\mu \in C^1(\R)$ satisfying
    $\mu(t)\gtrsim t\log(t)^{1+\delta}$, $\delta>0$,
			\begin{equation}\label{tesis1theorem4}
			\lim_{t \to \infty} \|u(t)\|_{H^1(|x|\sim \mu(t))}
      +\|u_t(t)\|_{L^2(|x|\sim \mu(t))}
      +\|n(t)\|_{L^2(|x|\sim \mu(t))}+\|v(t)\|_{L^2(|x|\sim \mu(t))} = 0.
	\end{equation}
\end{theorem}

\subsection*{Notation}

  We introduce
 	\begin{equation}\label{WeightedNorms}
    \begin{gathered}
 	    \|u(t)\|_{L^2_{\omega}(\R)}^2 := \intR \omega(x) |u(t,x)|^2dx,\\
 	    \|u(t)\|_{H^1_{\omega}(\R)}^2 := \intR \omega (x) \left(|u_x(t,x)|^2
      + |u(t,x)|^2\right)dx,
    \end{gathered}
 	\end{equation}
 	as the weighted $L^2$-norm and $H^1$-norm.

\medskip

This paper is organized as follows. In Section \ref{DecayOnCompactIntervals} we
 prove Theorem \ref{theorem1}; the virial argument is given in Subsection
 \ref{VirialArgument}. Section \ref{DecayinFarRegions} is devoted to the proof
 of Theorem \ref{theorem2}. Sections \ref{KGOnCompactIntervals} and
 \ref{KGFarFieldRegions} contain the KGZ system results, Theorems \ref{theorem3}
  and \ref{theorem4}, respectively.

\section{Decay on compact intervals for Zakharov}\label{DecayOnCompactIntervals}

This section is devoted to the proof of Theorem \ref{theorem1}. Before we begin
with the virial analysis, we give the following result, which states boundness
of the energy norm for every solution to \eqref{ZSb} with finite energy, It will
 be useful also in Section \ref{DecayinFarRegions}.

\begin{lemma}\label{LemmaEstimationOfTermsOnTheEnergy}
	Let $(u, n, v) \in C\left(\R^+, H^1(\R)\times L^2(\R)\times L^2(\R)\right)$
	be a solution of \eqref{ZSb} such that $E_s<\infty$ and $M_s<\infty$. Then, there exists
  $K_s>0$ ($K_s$ depending only on the initial data) such that
		\begin{equation}\label{E4}
		\intR \left(|u_x(t,x)|^2+ |u(t,x)|^2 + |v(t,x)|^2+|n(t,x)|^2 \right) dx
    \le K_s.
		\end{equation}
\end{lemma}

\begin{proof}
  We have that
		\[
    \begin{aligned}
		  \intR |u_x|^2 + \frac12 \left( |v|^2+|n|^2 \right) dx
		  &=	\intR |u_x|^2 + \frac12 \left( |v|^2+|n|^2 \right) +2n|u|^2 dx
          - 2 \intR n|u|^2dx\\
      &\le \intR |u_x|^2 + \frac12 \left( |v|^2+|n|^2 \right) +2n|u|^2 dx
          + 2\intR |n||u|^2dx.
    \end{aligned}
    \]
  Using Young inequality for products, for $\epsilon>0$ we get
		\begin{equation}\label{e1}
      \begin{aligned}
      &\intR |u_x|^2 + \frac12 \left( |v|^2+|n|^2 \right) dx\\
		  &\le \intR |u_x|^2 + \frac12 \left( |v|^2+|n|^2 \right) +2n|u|^2 dx
            +  \frac{1}{\epsilon }
      \intR |n|^2dx+ {\epsilon}\intR |u|^4dx.
    \end{aligned}
		\end{equation}
	Now, Gagliardo-Nirenberg inequality \cite{KV, Nagy}, implies that
		\[
		  \intR |u|^4dx \le  C_{GN} \|u_x\|_{L^2(\R)}\|u\|_{L^2(\R)}^{3}
      \le \frac {C_{GN}}{2}\|u_x\|_{L^2(\R)}^2
          + \frac {C_{GN}}{2}\|u\|_{L^2(\R)}^{6},
		\]
  where
    \[
    C_{GN}
   = \frac{\sqrt{3}}{3},
    \]
  and $Q$ is the solution to $Q''+Q^3-Q=0$.
	Then, going back to \eqref{e1} and taking, for instance, $\epsilon=2$, we
  obtain
		\begin{equation*}
      \intR |\nabla u|^2 + \frac12 \left( |v|^2+|n|^2 \right) dx
		  \le 2 E_s(0) + \frac{\sqrt3}{6} M_s(0)^3
		\end{equation*}
	Which means that there exists a constant $K_s$ depending on $M_s(0)$ and
  $E_s(0)$ such that \eqref{E4} holds.
\end{proof}

\subsection{Virial argument}\label{VirialArgument}
\par \emph{Step 1: Virial identity.}
\\
We now  introduce suitable virial identities that allow us to work out our
argument. Let $\varphi \in C^\infty(\R)$ be a bounded real function. Define
	\[
	I(t)= \im \intR \varphi(x) u(t,x)\overline{u}_x(t,x) dx
      - \intR \varphi(x) v(t,x)n(t,x) dx.
	\]
Then, we have the following:
\begin{lemma}[Virial identity]
Let $(u,n,v)$ be a solution to \eqref{ZSb}. Then,
		\begin{equation}\label{virial}
		\begin{aligned}
		- \deri I(u(t))&= 2\intR \varphi'(x) \left|u_x(t,x)\right|^2 dx
    - \frac12 \intR  \varphi'''(x) |u(t,x)|^2 dx
    + \intR \varphi'(x) n(t,x) |u(t,x)|^2\\
		& \quad {}~ +\frac12 \intR \varphi'(x) |n(t,x)|^2  dx
		+ \frac12\intR \varphi' |v(t,x)|^2 dx.
		\end{aligned}
		\end{equation}
\end{lemma}

\begin{proof}
	We compute
		\[
		\begin{aligned}
		\deri I(u(t)) =
    &   \im \intR \varphi(x) u_t(t,x)\overline{u}_x(t,x) dx
      + \im \intR \varphi(x) u(t,x)\overline{u}_{tx}(t,x) dx \\
    & - \intR \varphi(x) v_t(t,x)n(t,x) dx
		  - \intR \varphi(x) v(t,x)n_t(t,x) dx.
		\end{aligned}
		\]
	Integrating by parts,
		\[
		\begin{aligned}
		\deri I(t)
		& = 2 \im \intR \varphi(x) u_t(t,x)\overline{u}_x(t,x) dx
    -\im \intR \varphi'(x) u(t,x)\overline{u}_{t}(t,x) dx\\
    & \quad {}~ - \intR \varphi(x) v_t(t,x)n(t,x) dx
		- \intR \varphi(x) v(t,x)n_t(t,x) dx\\
		& = - 2 \re \intR \varphi(x) iu_t(t,x)\overline{u}_x(t,x) dx -
		\re \intR  \varphi'(x) u(t,x)\overline{i u_{t}}(t,x) dx\\
    & \quad {}~ - \intR \varphi(x) v_t(t,x)n(t,x) dx
		- \intR \varphi v(t,x)n_t(t,x) dx.
		\end{aligned}
		\]
	Now, since $(u, n,v)$ is a solution of the system \eqref{ZSb},
		\[
		\begin{aligned}
		  \deri I(t)
		  & = \intR \varphi(x) \left(\left|u_x(t,x)\right|^2\right)_x dx
        - \intR \varphi(x) n(t,x)\left(|u(t,x)|^2\right)_x dx
      + \re \intR  \varphi'(x) u(t,x)\overline{u}_{xx}(t,x) dx\\
		  &\quad {}~- \intR  \varphi'(x) |u(t,x)|^2n(t,x) dx
      + \intR \varphi(x) \left(|u(t,x)|^2+n(t,x)\right)_xn(t,x)dx\\
		  &\quad {}~+ \frac12\intR \varphi(x) \left(|v(t,x)|^2\right)_x dx.
		\end{aligned}
		\]
	We integrate by parts once more and get
		\[
		\begin{aligned}
		\deri I(t)
		& = - 2\intR \varphi'(x) \left|u_x(t,x)\right|^2 dx
    - \frac12 \intR  \varphi''(x) \left(|u(t,x)|^2\right)_x dx
    - \intR \varphi'(x) n(t,x) |u(t,x)|^2 \\
		&\quad {}~- \intR \varphi'(x) |n(t,x)|^2  dx
    + \frac12 \intR \varphi(x) \left(|n(t,x)|^2\right)_x dx
		- \frac12\intR \varphi'(x) |v(t,x)|^2 dx\\
		& = - 2\intR \varphi'(x) \left|u_x(t,x)\right|^2 dx
    + \frac12 \intR  \varphi'''(x) |u(t,x)|^2 dx
    - \intR \varphi'(x) n(t,x) |u(t,x)|^2\\
		& \quad{}~-\frac{1}{2} \intR \varphi'(x) |n(t,x)|^2  dx
		- \frac12\intR \varphi'(x) |v(t,x)|^2 dx.
		\end{aligned}
		\]
	Then, the identity follows.
\end{proof}

\medskip

\emph{Step 2: Estimations of the terms on the virial.}
\\
In order to find a more compact expression of \eqref{virial}, we define the
bilinear form
	\begin{equation}\label{Bilineal1}
	B(u)=2\intR \varphi' \left|u_x\right|^2 dx
      - \frac12\intR \varphi''' |u|^2 dx.
	\end{equation}
Then identity \eqref{virial} turns into
	\begin{equation}\label{virialAfterBilineal}
	- \deri I(u(t))= B(u)+ \intR	 \varphi' n |u|^2
  + \frac12 \intR \varphi' |n|^2  dx
	+ \frac12\intR \varphi' |v|^2 dx.
	\end{equation}
Following the argument in \cite{M}, for $\lambda>0$, let us take
$\varphi\left(x\right)= \lambda \tanh\left(\frac{x}{\lambda}\right)$
and $\omega(x)=\sqrt{\varphi'(x)}$. Notice that if $u=u_1+iu_2$, where $u_1$,
$u_2$ are real functions, then $B(u)=B(u_1) + B(u_2)$. We take $\eta=u_i$,
$i=1,2$, and find estimations for $B(\eta)$. From now on, we are going to assume
 $u$ odd, which implies that $\eta$ is also odd. Note that, by integration by
 parts,
	\begin{align*}
	\intR \left(\omega \eta\right)_x^2 dx
	& = \intR \varphi' \left({\eta}_x\right)^2 dx
    + \intR \omega \omega' \left(\eta^2\right)_x dx
    + \intR \left(\omega'\right)^2 \eta^2 dx\\
	& = \intR \varphi' \left({\eta}_x\right)^2 dx
    - \intR \omega \omega''\eta^2 dx,
	\end{align*}
It follows that
	\begin{align}\label{b11}
	\intR \varphi' \left({\eta}_x\right)^2 dx=
	\intR \left(\omega \eta\right)_x^2 dx
  + \intR  \frac{\omega''}{\omega}\left(\omega \eta\right)^2 dx.
	\end{align}
On the other hand, we can re-write
$\varphi'''=\left(\omega^2\right)''=2\left(\omega \omega''
+ (\omega')^2\right)$, and obtain
	\begin{align}\label{b12}
	\intR \varphi''' \eta^2 dx = 2\intR\left(\frac{\omega''}{\omega}
  +\frac{(\omega')^2}{\omega^2}\right)\left(\omega \eta \right)^2 dx.
	\end{align}
Thus, from (\ref{b11}) and (\ref{b12}),
	\begin{align*}
	B(\eta)=2\intR \left(\omega \eta\right)^2_x dx
  - \intR\left( \frac{(\omega')^2}{\omega^2}
  -\frac{\omega''}{\omega}\right)\left(\omega \eta \right)^2 dx.
	\end{align*}
Since $\omega(x)=\sech \left(\frac{x}{\lambda}\right)$, then
	\begin{align*}
	B(\eta) = 2\intR \left(\omega \eta\right)^2_x dx
  - \frac{1}{\lambda^2}
  \intR\sech^2\left(\frac{x}{\lambda}\right)\left(\omega \eta \right)^2 dx.
	\end{align*}
Now, introducing a new variable $\zeta=\omega\eta$, we set
	\[
	\mathcal{B}(\zeta) = 2\intR {\zeta}^2_x dx
  - \frac{1}{\lambda^2} \intR\sech^2\left(\frac{x}{\lambda}\right)\zeta^2 dx
	\]
so that
	\[
	\mathcal{B}(\zeta)=B(\eta).
	\]
At this point, we would like to prove that the bilinear $\mathcal{B}$ is coercive, which would imply
	\[
	B(\eta)=\mathcal{B}(\zeta)\gtrsim \|\zeta_x\|^2_{L^2(\R)}=\|\left(\omega \eta\right)_x\|^2_{L^2(\R)}
	\]
 That way, we could have an estimation for the bilinear part on \eqref{virialAfterBilineal}, $B(u)$, using the weighted norm $\|\cdot\|_{H^1_{\omega}}$.

\begin{lemma}[See \cite{M}]\label{Coercivity}
	Let $\zeta \in H^1(\R)$ be odd. Then,
		\begin{equation*}
		\mathcal{B}(\zeta)\ge \frac32 \intR \zeta_x^2 dx.
		\end{equation*}
\end{lemma}

We refer to \cite[Proposition 2.2]{M} for the details of the proof.

\medskip

\par Finally, to conclude the analysis of the linear term $B(u)$, we need to bound this term by  $\|u\|_{H^1_\omega(\R)}$.

\begin{lemma}[See \cite{M}]\label{CoercivitySecond}
	Let $u \in H^1(\R)$ be odd, $u=u_1+iu_2$. Then there exists a positive constant $c_0<1$ such that
		\begin{equation}\label{E1}
		B(u)\ge c_0\|u\|^2_{H^1_\omega(\R)}
		\end{equation}
\end{lemma}
We omit the proof, see \cite[Lemma 2.3]{M}.

 \medskip

 \emph{Step 3: Conclusion of the argument.}
 \\
 The key ingredient of the virial argument is the subsequent proposition:

 \begin{proposition}\label{keyproposition}
 	Let $(u, n, v)$ be a solution of \eqref{ZSb} such that $u$ is odd and satisfies \eqref{hyp}, for $\varepsilon>0$ sufficiently small.
  Then, there exists $C>0$ such that
 	  \begin{equation}\label{E2}
 	    \intp \|u(t)\|_{H^1_w(\R)}^2 +\frac12 \|v(t)\|_{L^2_w(\R)}^2
      +\frac12 \|n(t)\|_{L^2_w(\R)}^2 dt \le C.
 	  \end{equation}
 	In particular,
 	  \begin{equation}\label{E3}
 	    \intp \|u_x(t)\|_{L^2_w(\R)}^2 +\frac12 \|v(t)\|_{L^2_w(\R)}^2
      +\frac12 \|n(t)\|_{L^2_w(\R)}^2 dt \le C.
 	  \end{equation}
 \end{proposition}

\begin{proof}
	From \eqref{virialAfterBilineal} and \eqref{E1}, we have that
    \begin{equation}\label{FirstEquationProofKeyNLSZakharov}
      -\deri I(t)\ge c_0\|u(t)\|^2_{H^1_\omega(\R)} +
      \frac12 \intR\varphi'\big(v^2+n^2\big)dx
      + \intR \varphi' n |u|^2 dx.
    \end{equation}
  Now, following the idea of the proof of Lemma
  \ref{LemmaEstimationOfTermsOnTheEnergy}, by
  Young inequality, for some $\epsilon>0$ we
  get
    \begin{equation}\label{SecondEquationProofKeyNLSZakharov}
    \begin{aligned}
      \intR \varphi' |n| |u|^2 dx
      &\le \frac{1}{2\epsilon}\intR \varphi' n^2 dx
          +\frac{\epsilon}{2} \intR\varphi' |u|^4 dx.
    \end{aligned}
    \end{equation}
At this point, we need to absorb the negative terms using the weighted-norm
\eqref{WeightedNorms}.
		Since $u$ is odd,
		\begin{align*}
		\intR \sech^2\left(\frac{x}{\lambda}\right)|u|^{4}dx
		& =2 \intp \sech^2\left(\frac{x}{\lambda}\right)|u|^{4}dx \\
		& = 2 \intp	\sech^{-2}\left(\frac{x}{\lambda}\right)
    \sech^{4}\left(\frac{x}{\lambda}\right)|u|^{4} \\
		& \simeq \intp  e^{2x/\lambda}
		\sech^{4}\left(\frac{x}{\lambda}\right)|u|^{4}dx.
		\end{align*}
	\par With a slight abuse of notation, set
  $\zeta(t,x):=\sech\left(\frac{x}{\lambda}\right)u(t,x)$.
	Note that $\zeta(t,0)=0$ and vanishes at infinity $\forall t \in \R$.
  Then, integrating by parts,
		\begin{align*}
		\intp  e^{2x/\lambda} |\zeta|^{4}dx
		&=-\frac{\lambda}{2} \intp  e^{2x/\lambda}
		\left(|\zeta|^{4}\right)_x dx \\
		&=- 2\lambda\re
    \intp e^{2x/\lambda} |\zeta|^{2} \bar \zeta{\zeta}_x dx.
		\end{align*}
	Hence,
		\begin{align*}
		\intp  e^{2x/\lambda} |\zeta|^{4}dx
		& =-2\lambda\re
      \intp e^{x/\lambda} |\zeta|
      \overline{\zeta}\zeta_x\left(e^{x/\lambda}
      |\zeta|\right) dx \\
		& \lesssim \|u\|_{L^{\infty}(\R)} \re \intp e^{x/\lambda}
      |\zeta|\overline{\zeta}\zeta_x dx \\
		&\lesssim \|u\|_{L^{\infty}(\R)} \intp e^{x/\lambda}
      |\zeta|^2|\zeta_x| dx.
		\end{align*}
	By Young's inequality,
		\begin{align*}
		\intp  e^{2x/\lambda}
		|\zeta|^{4}dx & \lesssim \|u\|_{L^{\infty}(\R)}
    \intp |\zeta_x|^2 dx
    +\|u\|_{L^{\infty}(\R)}\intp e^{2x/\lambda} |\zeta|^{4}dx \\
		&\simeq \|u\|_{L^{\infty}(\R)} \intp |\zeta_x|^2 dx
      +\|u\|_{L^{\infty}(\R)}\intp \sech^{-2}\left(\frac{x}{\lambda}\right)
      \sech^{4}\left(\frac{x}{\lambda}\right)|u|^{4}dx \\
		& = \|u\|_{L^{\infty}(\R)}
    \intp |\zeta_x|^2dx
    +\|u\|_{L^{\infty}(\R)}
    \intp\sech^{2}\left(\frac{x}{\lambda}\right)|u|^{4}dx.
		\end{align*}
	By Sobolev's embedding and \eqref{hyp} with $0<\varepsilon<1$, this actually
  means that
		\[
    \intR \sech^2\left(\frac{x}{\lambda}\right)|u|^{4}dx
    \lesssim \varepsilon \intp |\left(\omega u\right)_x|^2 dx.
    \]
From Lema \ref{Coercivity} and Lema \ref{CoercivitySecond}, we obtain
  \[
  \intR \sech^2\left(\frac{x}{\lambda}\right)|u|^{4}dx
  \lesssim c_0\varepsilon \|u\|^2_{H^1_\omega}.
  \]
Then, going back to \eqref{FirstEquationProofKeyNLSZakharov}, taking  $\epsilon=2$ and $\varepsilon$ suficiently small, one gets
    \[
    \begin{aligned}
      -\deri I \gtrsim
      & \ \|u(t)\|^2_{H^1_\omega(\R)}+\|n(t)\|^2_{L^2_{\omega}(\R)}
      +\|v(t)\|^2_{L^2_{\omega}(\R)}.
    \end{aligned}
    \]
  Now, we integrate in time over $[0, \tau]$ for $\tau>0$,
    \[
    \begin{aligned}
      \int_0^\tau \|u(t)\|^2_{H^1_\omega(\R)}+\|n(t)\|^2_{L^2_\omega(\R)}
      +\|v(t)\|^2_{L^2_\omega(\R)}dt \lesssim |I(\tau)|+|I(0)|.
    \end{aligned}
    \]
  Thanks to Lemma \ref{LemmaEstimationOfTermsOnTheEnergy}, one obtains that
    \[
      |I(t)| \le \|u(t)\|^2_{H^1(\R)}+\|n(t)\|^2_{L^2(\R)}+
      \|v(t)\|^2_{L^2(\R)} \le K_s, \quad \forall t\ge 0.
    \]
  Finally, taking $\tau\to \infty$, we conclude.
\end{proof}

\medskip

\subsection{Proof of Theorem \ref{theorem1}:}

\par Now, we procede to conclude the proof of Theorem \ref{theorem1}.
\par Let $\phi \in C^\infty(\R)$. Using equation \eqref{ZSb}, we can compute
	\[
	\begin{aligned}
	\deri \intR \phi \left( |u|^2+|v|^2+|n|^2 \right) dx =
	& - 2\im \intR \phi' u\overline{u}_x dt
  - 2 \intR \phi v \left(|u|^2+n\right)_x dx - 2 \intR \phi n v_x dx.
	\end{aligned}
	\]
By integration by parts, one obtains
	\[
	\begin{aligned}
	&\deri \intR \phi \left( |u|^2+|v|^2+|n|^2 \right) dx =
  - 2\im \intR \phi' u\overline{u}_x dt
  + 2 \intR \phi' vn \ dx - 2 \intR \phi v \left(|u|^2\right)_x dx.
	\end{aligned}
	\]
This implies that
	\begin{equation}\label{VirialFromMass}
	\begin{aligned}
	\deri \intR \phi \left( |u|^2+|v|^2+|n|^2 \right) dx
	 =  -2\im\intR \phi' u\overline{u}_x dx
   + 2 \intR \phi' vn \ dx - 4 \re \intR \phi v u \overline u_x dx.
	\end{aligned}
	\end{equation}
From H\"older inequality and \eqref{hyp}, taking
$\phi(x)=\sech(x)$, we can conclude that
	\begin{equation}\label{e2}
	\begin{aligned}
	& \deri\left( \|u(t)\|_{L^2_w(\R)}^2 +\frac12 \|v(t)\|_{L^2_w(\R)}^2
  +\frac12 \|n(t)\|_{L^2_w(\R)}^2\right)\\
	& \le 2\intR \sech(x) \left(|u|^2 +  |{u}_x|^2 + |v|^2 + |n|^2 \right) dx
  +2 \varepsilon\intR \sech(x) \left(|v|^2 + |u_x|^2 \right) dx\\
	&\lesssim \|u(t)\|^2_{H^1_w(\R)}+\frac12\|n(t)\|^2_{L^2_w(\R)}
  +\frac12 \|v(t)\|^2_{L^2_w(\R)}.
	\end{aligned}
	\end{equation}
By \eqref{E3} we have that there exists a sequence $\{t_n\} \subset \R$, $t_n \to \infty$ such that
	\[
	\|u(t_n)\|_{L^2_w(\R)}^2 +\frac12 \|v(t_n)\|_{L^2_w(\R)}^2 +\frac12 \|n(t_n)\|_{L^2_w(\R)}^2 \to 0 .
	\]
We integrate \eqref{e2} over $[t, t_n]$, for some $t \in \R$ and take $t_n \to \infty$.
	\begin{equation*}
	\begin{aligned}
	&  \|u(t)\|_{L^2_w(\R)}^2 +\frac12 \|v(t)\|_{L^2_w(\R)}^2 +\frac12 \|n(t)\|_{L^2_w(\R)}^2 \\
	&\lesssim \int_t^\infty \|u(s)\|^2_{H^1_w(\R)}+\frac12\|n(s)\|^2_{L^2_w(\R)} +\frac12 \|v(s)\|^2_{L^2_w(\R)}ds.
	\end{aligned}
	\end{equation*}
Finally, taking $t \to \infty$, in view of \eqref{E3}, we obtain
	 \begin{equation}\label{DecayOfL2Norm}
	  \lim_{t \to \infty} \|u(t)\|_{L^2_w(\R)}^2 +   \|v(t)\|_{L^2_w(\R)}^2 + \|n(t)\|_{L^2_w(\R)}^2 =0.
    \end{equation}
  To show decay of the $L^\infty$-norm, we use the following claim, proven in \cite{M}:
  \begin{claim}\label{claim1}
  	For every interval $I$ there exists $\tilde{x}(t) \in I$ such that, as $t$ tends to infinity,
  		\[|
      u(t, \tilde{x}(t))|^2\to 0.
      \]
  \end{claim}
  Now, if $x \in I$, by Fundamental Theorem of calculus and H\"older's
  inequality
    \begin{align*}
  	|u(t,x)|^2 - |u(t, \tilde{x}(t))|^2
  	&= \int_{\tilde{x}(t)}^{x} \left(|u|^2\right)_x dx
      \le 2 \int_{\tilde{x}(t)}^{x} |u||u_x|dx\\
  	& \le 2 \|u(t)\|_{L^2(I)}\|u_x(t)\|_{L^2(I)}.
  	\end{align*}
  Then,
  	\begin{equation}\label{f1}
  	|u(t,x)|^2 \lesssim |u(t, \tilde{x}(t))|^2
    + 2 \|u(t)\|_{L^2(I)}\|u_x(t)\|_{L^2(I)},\quad \forall x \in I.
  	\end{equation}
  Using Lemma \ref{LemmaEstimationOfTermsOnTheEnergy}, we get
  	\[
     \sup_{t \in \R} \|u(t)\|_{H^1(\R)}< \infty.
  	\]
  Hence, taking $t \to \infty$ in \eqref{f1}, from Claim \ref{claim1} and
  \eqref{DecayOfL2Norm}, we get that
  	\[
    |u(t,x)|^2 \to 0, \quad \forall x \in I.
    \]

\section{Decay in regions along curves for Zakharov}\label{DecayinFarRegions}
From now on, let us assume $\lambda, \mu \in C^1(\R)$ are functions depending on time. For $\varphi \in C^2(\R)\cap L^\infty(\R)$, define
	\[
	K(t)=\frac{1}{2}\intR \varphi\left(\frac{x+\mu(t)}{\lambda(t)}\right)|u(t,x)|^2 dx,
	\]
and
	\[
	J(t)=\intR \varphi\left(\frac{x+\mu(t)}{\lambda(t)}\right) \left( |u_x(t,x)|^2+\frac12|v(t,x)|^2+\frac12 |n(t,x)|^2+n(t,x)|u(t,x)|^2+|u(t,x)|^2\right) dx.
	\]
As we did in the previous section, we obtain a virial identity from which we are going to construct the argument.

\begin{lemma} Let $(u,n,v) \in H^1\times L^2\times L^2$ a solution to \eqref{ZSb}. Then,
		\begin{enumerate}
			\item
			\begin{equation}\label{VirialFromMassTimeDependent}
			\begin{aligned}
			\dfrac{d}{dt}K(t) =
			& \dfrac{1}{\lambda(t)}\im \intR\varphi'\left(\dfrac{x+\mu(t)}{\lambda(t)}\right) \overline u(t,x) u_x(t,x) dx +\frac{\mu'(t)}{2\lambda(t)} \intR \varphi'\left(\frac{x+\mu(t)}{\lambda(t)}\right)|u(t,x)|^2dx\\ &-\frac{\lambda'(t)}{2\lambda(t)}\intR \varphi'\left(\frac{x+\mu(t)}{\lambda(t)}\right)\left(\dfrac{x+\mu(t)}{\lambda(t)}\right)|u(t,x)|^2dx.
			\end{aligned}
			\end{equation}

			\item
			 \begin{equation}\label{VirialFromTheEnergy}
			\begin{aligned}
			\dfrac{d}{dt}J(t)
      & = \frac{1}{\lambda(t)}\intR \varphi'\left(\frac{x+\mu(t)}{\lambda(t)}\right)\bigg(  \left(n +|u|^2\right)v\bigg)(t,x) dx\\
      &{} ~+  \frac{2}{\lambda(t)}\im\intR \varphi'\left(\frac{x+\mu(t)}{\lambda(t)}\right)\bigg( \overline u_x u_{xx} +n \overline u u_x + \overline u u_x\bigg)(t,x) dx\\
			&{}~ +\frac{\mu'(t)}{2\lambda(t)}\intR \varphi'\left(\frac{x+\mu(t)}{\lambda(t)}\right)\Big(2|u_x|^2 +2n|u|^2+ |u|^2+|v|^2+|n|^2\Big)(t,x) dx\\
			&{}~ -\frac{\lambda'(t)}{2\lambda(t)}\intR \varphi'\left(\frac{x+\mu(t)}{\lambda(t)}\right)\left(\frac{ x+\mu(t)}{\lambda(t)}\right)\Big(2|u_x|^2+2n|u|^2+ |u|^2+|v|^2+|n|^2\Big)(t,x) dx.
			\end{aligned}
			\end{equation}

\end{enumerate}
\end{lemma}

\begin{proof}
	The proof follows from \eqref{VirialFromMass}.
\end{proof}

Let us consider $\varphi \in C^2(\R)$ a decreasing bounded funtion such that $\varphi(s)=1$ for $s\le -1$ and $\varphi(s)=0$ for $s\ge 0$. Thus, we get that $\text{supp}(\varphi')\subset [-1,0]$ and
	\[
	\varphi'(s)\le 0, \quad \varphi'(s)s\ge 0 \quad \forall s \in \R.
	\]
Now, we want to take $\lambda$ such that $\lambda^{-1}$ integrates finite in time over an interval such as $[T, \infty]$, $T>0$. With this idea in mind, define, for $\delta>0$ and $t\ge 2$,  $\lambda(s)=t \log^{1+\delta}(t)$ and $\mu(t)=\lambda(t)$. In fact, we are considering $\mu=\lambda$ but in the following we will see that is possible to take $\mu(t)\gtrsim \lambda(t)$ and the computations would still work.
Then, for $t\ge 2$ we have
	\[
	\frac{\lambda'(t)}{\lambda(t)}=\frac{\mu'(t)}{\lambda(t)}\ge \frac{1}{t}+\frac{1+\delta}{t\log(t)}.
	\]
  \par So now, whereas $\lambda^{-1}(t)$ is integrable in $[2,\infty]$,
  $\frac{\lambda'}{\lambda}$ is not.

\subsection{First part of the proof of Theorem \ref{theorem2}}
\label{FirstPartOfTheProof}
In this subsection, we prove \eqref{tesis1theorem2}. The idea is to take
adventage of the non-integrability of $\frac{\lambda'}{\lambda}$ (and of
$\frac{\mu'}{\lambda}$, as well) by using the virial identity from $K(t)$. Let
us rearrange the terms in \eqref{VirialFromMassTimeDependent}:
	\begin{equation}\label{VirialFromMassTimeDependentRearreange}
	\begin{aligned}
	\frac{\lambda'(t)}{2\lambda(t)}
  \intR \varphi'\left(\frac{x+\mu(t)}{\lambda(t)}\right)
  \left(\dfrac{x+\mu(t)}{\lambda(t)}\right)|u(t,x)|^2dx
	-\frac{\mu'(t)}{2\lambda(t)}
  \intR \varphi'\left(\frac{x+\mu(t)}{\lambda(t)}\right)|u(t,x)|^2dx\\
	=-\dfrac{d}{dt}K(t)
  + \dfrac{1}{\lambda(t)}\im
  \intR\varphi'\left(\dfrac{x+\mu(t)}{\lambda(t)}\right)
  \overline u(t,x)u_x(t,x) dx
	\end{aligned}
	\end{equation}
Notice that, because of our election of $\varphi$, each term on the LHS is
positive. Now, our aim is to control the RHS so that we get integrability on
that part of the equation. Indeed, computing the integral of the RHS over
$[2, \infty]$, from Lemma \ref{LemmaEstimationOfTermsOnTheEnergy} we
have that
	\[
	\begin{aligned}
	\int_2^\infty \frac{1}{\lambda(\tau)}
  \im \intR \varphi'\left(\frac{x+\mu(t)}{\lambda(\tau)}\right)
  \overline{u}(\tau,x)u_x(\tau,x) dxd\tau
	&\le \int_2^\infty
  \frac{1}{\lambda(\tau)}\|u(\tau)\|_{L^2(\R)}\|u_x(\tau)\|_{L^2(\R)} d\tau\\
	&\lesssim \int_2^\infty \frac{1}{\lambda(\tau)} d\tau<\infty.\\
	\end{aligned}
	\]
 Thus, we integrate equation \eqref{VirialFromMassTimeDependentRearreange} in
 time and obtain:
	\begin{equation}\label{FiniteNonintegralPart}
	\int_2^\infty\frac{\lambda'(\tau)}{2\lambda(\tau)}
  \intR \left[\varphi'\left(\frac{x+\mu(\tau)}{\lambda(\tau)}\right)
  \left(\dfrac{x+\mu(\tau)}{\lambda(\tau)}\right)
	-
  \varphi'\left(\frac{x+\mu(\tau)}{\lambda(\tau)}\right)\right]|u(\tau,x)|^2
  dx d\tau
	<\infty.
	\end{equation}
This implies the existence of a sequence
$\{t_n\}\subset \R$, $t_n \to \infty$, such that
	\begin{equation}\label{TermsTendingToZero}
	\intR
	\left[ \varphi'\left(\frac{x+\mu(t_n)}{\lambda(t_n)}\right)
  \left(\dfrac{x+\mu(t_n)}{\lambda(t_n)}\right)-
  \varphi'\left(\frac{x+\mu(t_n)}{\lambda(t)}\right)\right]|u(t_n,x)|^2dx
  \to 0.
	\end{equation}
Now, take $\phi \in C_0^1(\R)$ such that
$\phi(s) \in [0,1] \text{ for all } s\in \R$,
$\text{supp}(\phi)=[-3/4, -1,4]$, satisfying
	\[
	\phi(s)\lesssim|\varphi'(s)|\text{ and } |\phi'(s)|\lesssim|\varphi'(s)|
  \text{ for all }s \in \R.
	\]
Then, if we consider $\phi$ instead $\varphi$ in
\eqref{VirialFromMassTimeDependent}, we obtain
	\begin{align*}
	&\left|\dfrac{d}{dt}\intR \phi\left(\frac{x+\mu(t)}{\lambda(t)}\right)|u(t,x)|^2dx\right|\\
	&{}~\lesssim \dfrac{1}{\lambda(t)} + \frac{\mu'(t)}{2\lambda(t)} \intR \left|\phi'\left(\frac{x+\mu(t)}{\lambda(t)}\right)-\phi'\left(\frac{x+\mu(t)}{\lambda(t)}\right)\left(\dfrac{x+\mu(t)}{\lambda(t)}\right)\right||u(t,x)|^2dx\\
	&{}~\lesssim \dfrac{1}{\lambda(t)} + \frac{\mu'(t)}{2\lambda(t)} \intR \left|\phi'\left(\frac{x+\mu(t)}{\lambda(t)}\right)\right||u(t,x)|^2dx,
	\end{align*}
Integrating over $[t, t_n]$ and taking $t_n \to \infty$, one gets from \eqref{TermsTendingToZero} that
	\begin{align*}
	\left|\intR \phi\left(\frac{x+\mu(t)}{\lambda(t)}\right)|u(t,x)|^2dx\right|\lesssim \int_t^{\infty}\dfrac{1}{\lambda(\tau)} d\tau
	+ \int_t^{\infty}\frac{\mu'(\tau)}{2\lambda(\tau)} \intR \left|\phi'\left(\frac{x+\mu(\tau)}{\lambda(\tau)}\right)\right||u(\tau,x)|^2dxd\tau.
	\end{align*}
Thus, because \eqref{FiniteNonintegralPart} holds, we can take $t \to \infty$ and conclude that
	\begin{align*}
	&\lim_{t \to \infty} \left|\intR \phi\left(\frac{x+\mu(t)}{\lambda(t)}\right)|u(t,x)|^2dx\right|\lesssim 0
	\end{align*}
Finally, the region of the convergence comes from the fact that $-\frac34\le\frac{x+\mu(t)}{\lambda(t)}-\frac14$ is equivalent to $x\sim \mu(t)$.

\subsection{Second part of the proof of Theorem \ref{theorem2}}\label{SecondPartOfTheProof}

This section deals with the proof of \eqref{tesis2theorem2}. The idea of the proof is the same as before, but since our aim is to show decay of the solution $(u,n,v)$, we need to consider the virial identity \eqref{VirialFromTheEnergy}. As before, we re-write the terms in \eqref{VirialFromTheEnergy} and get
	\begin{equation}\label{VirialFromTheEnergyRearrenged}
 	\begin{aligned}
 	&\frac{\lambda'(t)}{2\lambda(t)}\intR \varphi'\left(\frac{x+\mu(t)}{\lambda(t)}\right)\left(\frac{x+\mu(t)}{\lambda(t)}\right)\Big(2|u_x|^2+2n|u|^2+ |u|^2+|v|^2+|n|^2\Big)(t,x) dx\\
	&-\frac{\mu'(t)}{2\lambda(t)}\intR \varphi'\left(\frac{x+\mu(t)}{\lambda(t)}\right)\Big(2|u_x|^2 +2n|u|^2+ |u|^2+|v|^2+|n|^2\Big)(t,x) dx\\
	&=-\dfrac{d}{dt}J(t) +  \frac{1}{\lambda(t)}\im\intR \varphi'\left(\frac{x+\mu(t)}{\lambda(t)}\right)\bigg( 2\overline u_x u_{xx} + \overline u u_x \bigg)(t,x) dx\\
	&{} ~+  \frac{1}{\lambda(t)}\intR \varphi'\left(\frac{x+\mu(t)}{\lambda(t)}\right)\bigg(2\im\  n \overline u u_x+ \left(n +|u|^2\right)v\bigg)(t,x) dx.
 	\end{aligned}
	\end{equation}
  Just as before, we need to take $\lambda$ such that $\lambda^{-1}$
  integrates finite in time over an interval $[T, \infty]$, $T>0$. Then, for
  $\delta>0$ and $t\ge 2$, we define
  $\lambda(s)=t \log^{1+\delta}(t)f(t)$ and $\mu(t)=\lambda(t)$ (although
  the computations will still work for $\mu(t)\gtrsim \lambda(t)$).
  Then, for $t\ge 2$ we have
  	\[
  	\frac{\lambda'(t)}{\lambda(t)}=\frac{\mu'(t)}{\lambda(t)}\ge
    \frac{1}{t}+\frac{1+\delta}{t\log(t)}.
  	\]
  \par Notice that, if $\|u(t)\|_{H^2(\R)}\le C$ for all $t\ge 0$ for some
  $C>0,$ then we take $f(t)=C$ and get that
  $\lambda(t)^{-1}=\frac{1}{t\log^{1+\delta}(t)f(t)}
  \eqsim \frac{1}{t\log^{1+\delta}(t)}$.
  Furthermore, in the case $\|u(t)\|_{H^2(\R)}$ increasing infinitely, there
  would exist $T>0$ and $C>0$
  such that $\frac{1}{f(t)}\le C$ for $t\ge T$. Then, either way, we get that
  there exists $T\ge2$
    \[
    \lambda(t)^{-1}\lesssim\frac{1}{t\log^{1+\delta}(t)},\quad \text{for }t>T.
    \]
  Thus, we get that $\lambda^{-1}$ integrates finite over $[T, \infty)$,
  while $\frac{\lambda'}{\lambda}$ does not.
\par We estimate each term on the RHS of \eqref{VirialFromTheEnergy} after
integrating in time. From Lemma \ref{LemmaEstimationOfTermsOnTheEnergy}
we have that
	\[
	\begin{aligned}
	\int_T^\infty \frac{2}{\lambda(\tau)}\intR \varphi'\left(\frac{x+\mu(\tau)}{\lambda(\tau)}\right) \overline u_x(\tau,x) u_{xx}(\tau,x) dx d\tau
	& \lesssim \int_T^\infty \frac{1}{\lambda(\tau)} \|u_x(\tau)\|_{L^2(\R)}\|u_{xx}(\tau)\|_{L^2(\R)} d\tau \\
	& \lesssim  \int_T^\infty \frac{f(\tau)}{\lambda(\tau)} d\tau< \infty.
	\end{aligned}
	\]
The same way, we get
	\[
	\int_T^\infty \frac{2}{\lambda(\tau)}\intR \varphi'\left(\frac{x+\mu(\tau)}{\lambda(\tau)}\right) \overline u(\tau,x) u_{x}(\tau,x) dx d\tau
	<\infty.
	\]
Also, from Lemma \ref{LemmaEstimationOfTermsOnTheEnergy} and Sobolev embbedings,
	\[
	\begin{aligned}
	 &\int_T^\infty \frac{2}{\lambda(\tau)}\im \intR \varphi'\left(\frac{x+\mu(\tau)}{\lambda(\tau)}\right)n(\tau, x) \overline u(\tau, x) u_x(\tau, x) dx dt\\
	 &\lesssim \|u(t)\|_{L^{\infty}(\R)} \int_T^\infty \frac{2}{\lambda(\tau)}\intR|n(\tau, x)|^2 + |u_x(\tau, x)|^2 dx dt\lesssim \int_T^\infty \frac{2}{\lambda(\tau)}dt\le\infty,
	 \end{aligned}
	\]
	\[
	\int_T^\infty \frac{1}{\lambda(\tau)}\intR \varphi'\left(\frac{x+\mu(\tau)}{\lambda(\tau)}\right)n(\tau, x)v(\tau, x) dxdt
	\lesssim 	\int_T^\infty \frac{1}{\lambda(\tau)}\intR |n(\tau, x)|^2+|v(\tau, x)|^2 dxdt<\infty
	\]
and
	\[
	\int_T^\infty \frac{1}{\lambda(\tau)}\intR \varphi'\left(\frac{x+\mu(\tau)}{\lambda(\tau)}\right)|u(\tau, x)|^2v(\tau, x)dxdt \lesssim \|u(t)\|_{L^\infty(\R)}\int_T^\infty \frac{1}{\lambda(\tau)}\intR |u(\tau, x)|^2 + |v(\tau, x)|^2dx dt<\infty.
	\]
Consequently, integrating \eqref{VirialFromTheEnergyRearrenged} over $[T, \infty]$, one obtains
	\begin{equation}\label{FiniteTermsOfTheEnergyVirial}
	\begin{aligned}
	&\int_T^\infty\frac{\lambda'(\tau)}{2\lambda(\tau)}\intR \varphi'\left(\frac{x+\mu(\tau)}{\lambda(\tau)}\right)\left(\frac{x+\mu(\tau)}{\lambda(\tau)}\right)\Big(2|u_x|^2+2n|u|^2+ |u|^2+|v|^2+|n|^2\Big)(\tau,x) dx\\
	&\quad -\frac{\mu'(\tau)}{2\lambda(\tau)}\intR \varphi'\left(\frac{x+\mu(\tau)}{\lambda(\tau)}\right)\Big(2|u_x|^2 +2n|u|^2+ |u|^2+|v|^2+|n|^2\Big)(\tau,x) dx\ d\tau <\infty.
	\end{aligned}
	\end{equation}
Furthermore,
	\begin{equation}\label{FinitePositiveTermsOfTheEnergyVirial}
	\begin{aligned}
	&\int_T^\infty \frac{\mu'(\tau)}{2\lambda(\tau)}\intR \left|\varphi'\left(\frac{x+\mu(\tau)}{\lambda(\tau)}\right)\right|\Big(2|u_x|^2+ |u|^2+|v|^2+|n|^2\Big)(\tau,x) dxdt <\infty.
	\end{aligned}
	\end{equation}
Indeed, from \eqref{FiniteTermsOfTheEnergyVirial} and Young inequality for products,
	\[
	\begin{aligned}
	\infty
	>&\int_T^\infty\frac{\lambda'(\tau)}{2\lambda(\tau)}\intR \varphi'\left(\frac{x+\mu(\tau)}{\lambda(\tau)}\right)\left(\frac{x+\mu(\tau)}{\lambda(\tau)}\right)\Big(2|u_x|^2+2n|u|^2+ |u|^2+|v|^2+|n|^2\Big)(\tau,x) dx\\
	& -\frac{\mu'(\tau)}{2\lambda(\tau)}\intR \varphi'\left(\frac{x+\mu(\tau)}{\lambda(\tau)}\right)\Big(2|u_x|^2 +2n|u|^2+ |u|^2+|v|^2+|n|^2\Big)(\tau,x) dx\ dt\\
	\ge
  & \int_T^\infty\frac{\lambda'(\tau)}{2\lambda(\tau)}\intR \varphi'\left(\frac{x+\mu(\tau)}{\lambda(\tau)}\right)\left(\frac{x+\mu(\tau)}{\lambda(\tau)}\right)\Big(2|u_x|^2+ |u|^2+|v|^2+\frac12|n|^2-2|u|^4\Big)(\tau,x) dx\\
	& -\frac{\mu'(\tau)}{2\lambda(\tau)}\intR \varphi'\left(\frac{x+\mu(\tau)}{\lambda(\tau)}\right)\Big(2|u_x|^2 + |u|^2+|v|^2+\frac12|n|^2-2|u|^4\Big)(\tau,x) dx\ dt.
	\end{aligned}
	\]
By Sobolev embedding and \eqref{FiniteNonintegralPart}, we have that
	\[
  \begin{aligned}
	&\int_T^\infty \frac{\mu'(\tau)}{\lambda(\tau)}\intR \left(\varphi'\left(\frac{x+\mu(\tau)}{\lambda(\tau)}\right)\left(\frac{x+\mu(\tau)}{\lambda(\tau)}\right)-\varphi'\left(\frac{x+\mu(\tau)}{\lambda(\tau)}\right)\right)|u(\tau,x)|^4 dxdt\\
  &\lesssim\int_T^\infty \frac{\mu'(\tau)}{\lambda(\tau)}\intR \left(\varphi'\left(\frac{x+\mu(\tau)}{\lambda(\tau)}\right)\left(\frac{x+\mu(\tau)}{\lambda(\tau)}\right)-\varphi'\left(\frac{x+\mu(\tau)}{\lambda(\tau)}\right)\right)|u(\tau,x)|^2 dxdt<\infty.
\end{aligned}
  \]
Then, \eqref{FinitePositiveTermsOfTheEnergyVirial} holds. Thus, there exists a sequence $\{t_n\}$, $t_n\to \infty$ satisfying
	\begin{equation}\label{TermsTendingToZeroQuasiEnergy}
	\begin{aligned}
	\intR \left|\varphi'\left(\frac{x+\mu(t_n)}{\lambda(t_n)}\right)\right|\Big(|u_x|^2+ |u|^2+|v|^2+|n|^2\Big)(t_n,x) dx \to 0.
	\end{aligned}
	\end{equation}
Furthermore, using Sobolev embedding and Lemma \ref{LemmaEstimationOfTermsOnTheEnergy}, we have from \eqref{TermsTendingToZeroQuasiEnergy} that
    \[
    \begin{aligned}
    &\left|\displaystyle\intR \varphi'\left(\frac{x+\mu(t_n)}{\lambda(t_n)}\right)n(t_n,x)|u(t_n, x)|^2 dx \right|\\
    &\lesssim \intR \left|\varphi'\left(\frac{x+\mu(t_n)}{\lambda(t_n)}\right)\right||n(t_n,x)|^2 dx
    + \intR \left|\varphi'\left(\frac{x+\mu(t_n)}{\lambda(t_n)}\right)\right||u(t_n,x)|^2 dx \to  0.
  \end{aligned}
    \]
  Then,
    \begin{equation}\label{TermsTendingToZeroEnergy}
	     \begin{aligned}
         \left|
          \intR \varphi'\left(\frac{x+\mu(t_n)}{\lambda(t_n)}\right)\Big(|u_x|^2+ |u|^2+|v|^2+|n|^2 +n|u|^2\Big)(t_n,x) dx\right| \to 0.
	      \end{aligned}
	  \end{equation}
We argue as before and consider $\phi \in C_0^1(\R)$ such that $\phi(s) \in [0,1] \text{ for all } s\in \R$, $\text{supp}(\phi)=[-3/4, -1,4]$,
\[
\phi(s)\lesssim|\varphi'(s)|\text{ and } |\phi'(s)|\lesssim|\varphi'(s)| \text{ for all }s \in \R.
\]
One gets,
	\begin{equation*}
	\begin{aligned}
	& \left|\dfrac{d}{dt}\intR \phi\left(\frac{x+\mu(t)}{\lambda(t)}\right)\bigg(|u_x|^2+\frac12 |u|^2+\frac12 |n|^2+\frac12|v|^2+n|u|^2\bigg)(t,x)dx \right|\\
	&\quad \lesssim  \frac{2}{\lambda(t)}\intR\left| \phi'\left(\frac{x+\mu(t)}{\lambda(t)}\right)\right|\left|\bigg( \overline u_x u_{xx} +n \overline u u_x\bigg)(t,x)\right| dx\\
	&\quad {}~+  \frac{1}{\lambda(t)}\intR \left|\phi'\left(\frac{x+\mu(t)}{\lambda(t)}\right)\right|\left|\bigg( \overline u u_x + \left(n +|u|^2\right)v\bigg)(t,x)\right| dx\\
	&\quad {}~ +\frac{\mu'(t)}{2\lambda(t)}\intR \left|\phi'\left(\frac{x+\mu(t)}{\lambda(t)}\right)\right|\left|\Big(2|u_x|^2 +2n|u|^2+ |u|^2+|v|^2+|n|^2\Big)(t,x)\right| dx\\
	&\quad \lesssim  \frac{1}{\lambda(t)} +\frac{\mu'(t)}{2\lambda(t)}\intR \left|\phi'\left(\frac{x+\mu(t)}{\lambda(t)}\right)\right|\left|\Big(2|u_x|^2 +2|n|^2 + |u|^4+ |u|^2+|v|^2\Big)(t,x)\right| dx\\
	\end{aligned}
	\end{equation*}

Integrate over $[t, t_n]$, $t \ge T$
%
and take $t_n\to \infty$. Thanks to \eqref{TermsTendingToZeroEnergy}, we obtain
  \begin{equation*}
    \begin{aligned}
      & \left|\intR \phi\left(\frac{x+\mu(t)}{\lambda(t)}\right)\bigg(|u_x|^2+\frac12 |u|^2+\frac12 |n|^2+\frac12|v|^2+n|u|^2\bigg)(t,x)dx \right|\\
      &\lesssim \int_t^{\infty} \frac{1}{\lambda(\tau)} dt +\int_t^{\infty} \frac{\mu'(\tau)}{2\lambda(\tau)}\intR \left|\phi'\left(\frac{x+\mu(\tau)}{\lambda(\tau)}\right)\right|\left|\Big(2|u_x|^2 +|n|^2 + |u|^4+ |u|^2+|v|^2\Big)(\tau,x)\right| dx dt
    \end{aligned}
  \end{equation*}
Now, as in subsection \ref{FirstPartOfTheProof}, taking $t\to \infty$, we obtain
  \begin{equation*}
    \begin{aligned}
      \lim_{t\to \infty} \left|\intR \phi\left(\frac{x+\mu(t)}{\lambda(t)}\right)\bigg(|u_x|^2+\frac12 |u|^2+\frac12 |n|^2+\frac12|v|^2+n|u|^2\bigg)(t,x)dx \right|\le 0
    \end{aligned}
  \end{equation*}
Note that by H\"older inequality and Lemma \ref{LemmaEstimationOfTermsOnTheEnergy},
  \[
    \begin{aligned}
      \intR \phi\left(\frac{x+\mu(t)}{\lambda(t)}\right)\left(n|u|^2\right)(t,x)dx & \lesssim \|n(t)\|_{L^2(\R)} \left(\intR \phi\left(\frac{x+\mu(t)}{\lambda(t)}\right)|u(t,x)|^2dx\right)^{\frac12}\\
      &\lesssim \left(\intR \phi\left(\frac{x+\mu(t)}{\lambda(t)}\right)|u(t,x)|^2dx\right)^{\frac12}.
    \end{aligned}
  \]
Thanks to \eqref{tesis1theorem2}, we conclude the proof.

\section{Decay on compact intervals for Klein-Gordon-Zakharov }\label{KGOnCompactIntervals}

Before we present the proof of Theorems \ref{theorem3} and \ref{theorem4},
we give an estimation of the energy norm of a solution for \eqref{KGZSv},
that will be useful in the following.

\begin{lemma}\label{PropositionFiniteNormsKGZ}
Let $(u, u_t, n, v) \in C\left(\R^+, H^1(\R)\times L^2(\R)\times L^2(\R)\times L^2(\R)\right)$
be a solution of \eqref{ZSb} such that $E_{KG}<\infty$  and it satisfies \eqref{KGhyp} for some $C>0$ and $\varepsilon>0$ not necessarily small. Then, there exists $K_{KG}>0$ such that
  \[
    \intR |u_t(t,x)|^2 + |u_x(t,x)|^2 + |u(t,x)|^2 + |n(t,x)|^2 + |v(t,x)|^2 dx
    \le K_{KG}.
  \]
\end{lemma}

\begin{proof}
  We write
    \[
    \begin{aligned}
      &\frac12\intR |u_t|^2+|u_x|^2+|u|^2+|n|^2+|v|^2 dx \\
      &\le \intR |u_t|^2+\frac12|u_x|^2+|u|^2+\frac12|n|^2+|v|^2 + 2n |u|^2dx
      -2 \intR n|u|^2dx.
    \end{aligned}
    \]
  The first integral in the RHS can be bounded by the energy. Thus, we need
  to control the remaining term. By Young inequality and Gagliardo-Nirenberg
  inequality \cite{KV, Nagy}, for $\epsilon>0$ we have that
    \[
    2\intR n u^2 dx \le \frac1\epsilon \intR n^2 dx + \epsilon\intR u^4 dx
    \le \frac1\epsilon \intR n^2 dx +
    \epsilon C_{GN} \|u_x\|_{L^2(\R)}\|u\|^3_{L^2(\R)},
    \]
  where
    \[
    C_{GN}= \frac{\sqrt{3}}{3}
    \]
  and $Q$ is the solution to $Q''+Q^3-Q=0$.
  Then, taking $\epsilon=2$, we get
    \[
    2\intR n u^2 dx
    \le \frac 12 \intR n^2 dx + 2\frac{\sqrt{3}}{3}
    \|u_x\|_{L^2(\R)}\|u\|^3_{L^2(\R)}\le
    \frac 12 \intR n^2 dx + \frac{\sqrt{3}}{3}
    \intR u_x^2dx + \frac{\sqrt{3}}{3}\|u\|^6_{L^2(\R)}.
    \]
  Finally, this means that
    \[
    \begin{aligned}
      &\frac12\intR |u_t|^2+|u_x|^2+|u|^2+|n|^2+|v|^2 dx \le 2E_0 +
      \frac{\sqrt{3}}{3}\|u\|^6_{L^2(\R)}.
    \end{aligned}
    \]
  Thanks to \eqref{KGhyp}, we conclude.
\end{proof}

\subsection{Virial argument}
During this section, we are going to consider $(u, n, v)$ a solution such that $u$ is odd and satisfies \eqref{KGhyp}, for some $C>0$ and $\varepsilon$ small.
As in Section \ref{DecayOnCompactIntervals}, let $\varphi \in C^\infty(\R)$ a
 bounded real function and define
  \[
  I(t)=2\intR \varphi(x)u_x(t,x)u_t(t, x) dx
   - \intR\varphi(x) v(t,x)n(t,x) dx+\intR \varphi'(x)u(t,x)u_t(t,x) dx.
  \]
We get the following virial identity:
\begin{lemma}[Virial Identity]\label{VirialIdentityKG}
	Let $(u, u_t, n, v)$ be a solution to \eqref{KGZSv}. Then,
    \begin{equation}
      - \deri I(t) = 2\intR \varphi' u_x^2 dx - \frac12 \intR \varphi'''u^2 dx
      + \frac12 \intR \varphi' n^2 dx + \frac12 \intR\varphi' v^2 dx
      + \intR \varphi'u^2n dx.
    \end{equation}
\end{lemma}

\begin{proof}
  Using equation \eqref{KGZSv}, we compute
    \[
    \begin{aligned}
      \deri I(t)=
      & 2\intR \varphi u_{xt}u_t dx +2 \intR \varphi u_{x}u_{xx} dx
        - 2 \intR \varphi u_{x}u dx -2\intR \varphi n u_{x}u dx
        + \intR \varphi \left(n+u^2\right)_x n dx \\
      & + \intR\varphi v v_x dx+2\intR \varphi'u_tu_tdx
        + \intR \varphi'uu_{xx} dx -\intR \varphi'uu dx
        - \intR \varphi'uun dx \\
      =
      & \intR \varphi (u_t^2)_x dx + \intR \varphi (u_x^2)_x dx
        - \intR \varphi (u^2)_x dx - \intR \varphi n (u^2)_x dx
        + \frac12 \intR \varphi (n^2)_x dx \\
      & + \intR \varphi \left(u^2\right)_x n dx
        + \frac12 \intR\varphi (v^2)_x dx + \intR \varphi'u_t^2 dx
        + \intR \varphi'uu_{xx} dx -\intR \varphi'u^2 dx -\intR \varphi'u^2n dx.
    \end{aligned}
    \]
  We integrate by parts
    \[
    \begin{aligned}
      \deri I(t)
      =
      &   - 2\intR \varphi' u_x^2 dx
        - \frac12 \intR \varphi' n^2 dx
        - \frac12 \intR\varphi' v^2 dx
        - \intR \varphi''uu_{x} dx
        - \intR \varphi'u^2n dx\\
      =
      & - 2\intR \varphi' u_x^2 dx + \frac12 \intR \varphi'''u^2 dx
        - \frac12 \intR \varphi' n^2 dx - \frac12 \intR\varphi' v^2 dx
        - \intR \varphi'u^2n dx.
    \end{aligned}
    \]
\end{proof}

Notice that now we have a very similiar virial identity to the one obtained in
Subsection \ref{VirialArgument}. In fact, the RHS is the same.
Then, we are
entitled to use the estimations for the bilinear part of \eqref{virial}. Indeed,
we can write
  \[
  -\deri I(t) =  B(u) + \frac12 \intR \varphi' n^2 dx
  + \frac12 \intR\varphi' v^2 dx + \intR \varphi'u^2n dx,
  \]
where $B$ is defined in \eqref{Bilineal1}. Just as before, for $\lambda>0$,
consider $\varphi(x)=\lambda \tanh(x/\lambda)$ and
$\omega(x)=\sqrt{\varphi'(x)}$. Finally, using the arguments in
Subsection \ref{VirialArgument} and by estimation \eqref{E1}, we have that
  \begin{equation}\label{VirialIdentityWithEstimationsKGZ}
  -\deri I(t) \gtrsim \|u(t)\|_{H^1_\omega(\R)}^2
  + \frac12 \intR \varphi' n^2 dx
  + \frac12 \intR\varphi' v^2 dx + \intR \varphi'u^2n dx,
  \end{equation}
where $\|\cdot \|_{H^1_\omega(\R)}$ is the weighted-norm introduced in
\eqref{WeightedNorms}.

\medskip

\par In order to conclude the argument, we present the following proposition:

\begin{proposition}\label{KeyKGZakharov}
Let $(u,n,v)$ be a solution of \eqref{KGZSv}. Then, there exists $C>0$ such that
  \[
    \intp \|u_t(t)\|_{L^2_\omega(\R)}+\|u(t)\|_{H^1_\omega(\R)}+
    \|n(t)\|_{L^2_\omega(\R)}
    +\|v(t)\|_{L^2_\omega(\R)}dt\le C.
  \]
\end{proposition}

\begin{proof}
Thanks to \eqref{VirialIdentityWithEstimationsKGZ} and \eqref{KGhyp}, the proof follow as in Proposition \ref{keyproposition}.
\end{proof}

\subsection{Conclusion of the proof}

Let $\phi$ be a $C^\infty(\R)$ function to be defined later. Then, since $(u,n,v)$ is a solution to equation \eqref{KGZSv},
  \[
    \begin{aligned}
    &\deri \frac12 \intR \phi(x)\left(|u|^2+|u_t|^2+|u_x|^2+|n|^2+|v|^2 \right)(t,x)dx \\
    &=\intR \phi(x)u_t(t,x)u_{xx}(t,x)dx-\intR\phi(x)n(t,x)u_t(t,x)u(t,x)dx \\
    &\quad \ ~ + \intR \phi(x) u_x(t,x)u_{xt}(t,x)dx-\intR \phi(x)n(t,x)v_x(t,x)dx-\intR\phi(x) v(t,x)\left(n+|u|^2\right)_x(t,x)dx.
    \end{aligned}
  \]
After integration by parts, one gets
  \begin{equation}\label{VirialToCOnclude}
    \begin{aligned}
    & \deri \frac12 \intR
      \phi(x)\left(|u|^2+|u_t|^2+|u_x|^2+|n|^2+|v|^2 \right)(t,x)dx \\
    =
    & -\intR \phi'(x)u_t(t,x)u_{x}(t,x)dx-
        \intR\phi(x)n(t,x)u_t(t,x)u(t,x)dx
      + \intR\phi'(x) v(t,x)n(t,x)dx\\
    & + 2\intR\phi(x)v(t,x)u(t,x)u_x(t,x)dx.
    \end{aligned}
  \end{equation}
Thus, if we take $\phi(x)=\sech(x)$, we have that
  \[
    \begin{aligned}
    &\deri \frac12 \intR \sech(x)\left(|u|^2+|u_t|^2+|u_x|^2+|n|^2+|v|^2 \right)(t,x)dx\\ &\quad \lesssim  \|u(t)\|^2_{H^1_\omega(\R)}+\|u_t(t)\|^2_{L^2_\omega(\R)} +\|v(t)\|^2_{L^2_\omega(\R)} +\|n(t)\|_{L^2(\R)}^2.
    \end{aligned}
  \]
Proposition \ref{KeyKGZakharov} implies the existence of a sequence $\{t_n\}\subset \R$, $t_n \to \infty$, such that
  \[
    \|u(t_n)\|^2_{H^1_\omega(\R)}+\|u_t(t_n)\|^2_{L^2_\omega(\R)} +\|v(t_n)\|^2_{L^2_\omega(\R)} +\|n(t_n)\|_{L^2(\R)}^2\to 0.
  \]
Then, we integrate over $[t, t_n]$, take $t_n \to \infty$ and obtain
  \[
    \begin{aligned}
     &\|u(t)\|^2_{H^1_\omega(\R)}+\|u_t(t)\|^2_{L^2_\omega(\R)} +\|v(t)\|^2_{L^2_\omega(\R)} +\|n(t)\|_{L^2(\R)}^2\\
     & \quad \lesssim \int_t^\infty  \|u(\tau)\|^2_{H^1_\omega(\R)}+\|u_t(\tau)\|^2_{L^2_\omega(\R)} +\|v(\tau)\|^2_{L^2_\omega(\R)} +\|n(\tau)\|_{L^2(\R)}^2 d\tau.
    \end{aligned}
  \]
Thanks to Proposition \ref{KeyKGZakharov}, the RHS of the last equation is finite. Consequently, we can take $t \to \infty$ and conclude the proof.

\section{Decay in regions along curves for Klein-Gordon-Zakharov}\label{KGFarFieldRegions}

To construct the virial identity, consider $\varphi \in C^2(\R)$ a bounded real function and $\lambda, \mu \in C^1(\R)$ functions depending on time. We define
  \[
    J(t)=\frac12 \intR \varphi\left(\frac{x+\mu(t)}{\lambda(t)}\right)
    \bigg(|u_x|^2
    +|u|^2 + |u_t|^2 + \frac{1}{2} |v|^2
    + \frac12 |n|^2 + n|u|^2 \bigg)(t,x) dx.
  \]

\begin{lemma}\label{SecondVirialKG}
  Let $(u,n,v)$ be a solution to \eqref{KGZSv}. Then,
  \begin{equation*}
    \begin{aligned}
    \deri J(t)
    &
    = \frac{1}{\lambda(t)}\intR \varphi'\left(\frac{x+\mu(t)}{\lambda(t)}\right)
    \bigg(\frac{1}{2} vn+ \frac12 v|u|^2 - u_t u_{x}
      \bigg)(t,x) dx\\
    &
    +\frac{\mu'(t)}{2\lambda(t)}\intR \varphi'\left(\frac{x+\mu(t)}{\lambda(t)}\right)\bigg(|u_x|^2
    +|u|^2 + |u_t|^2 + \frac12 |v|^2 + \frac12 |n|^2 + n|u|^2 \bigg)(t,x) dx\\
    &
    -\frac{\lambda'(t)}{2\lambda(t)}\intR \varphi'\left(\frac{x+\mu(t)}{\lambda(t)}\right)
    \left(\frac{x+\mu(t)}{\lambda(t)}\right)\bigg(|u_x|^2
    +|u|^2 + |u_t|^2 + \frac12 |v|^2 + \frac12 |n|^2 + n|u|^2 \bigg)(t,x) dx.
    \end{aligned}
  \end{equation*}
\end{lemma}

\medskip

\noindent We skip the proof of Lemma \ref{SecondVirialKG}, since it follows from \eqref{VirialToCOnclude}

\subsection{Proof Theorem \ref{theorem4}}

As we did in Section \ref{DecayinFarRegions}, we consider $\varphi \in C^2(\R)$ a decreasing function satisfying
$\varphi(s)=1$ for $s\le -1$ and $\varphi(s)=0$ for $s\ge 0$. It follows that $\text{supp}(\varphi')\subset [-1, 0)$ and
  \[
    \varphi'(s)\le 0, \quad \varphi'(s)s\ge 0 \quad
    \forall s \in \R.
  \]
Also, for $\delta>0$, we take $\lambda(t)=t\log^{1+\delta}(t)$ and $\mu(t)=\lambda(t)$. Just as before, we are going to take adventage of the fact that $\lambda^{-1}$ is integrable in time over the time interval $[T, \infty)$, for some $T\ge 2$, while $\frac{\lambda'}{\lambda}$ is not.

\medskip

We re-arrange the virial identity \eqref{SecondVirialKG} as:
  \begin{equation}\label{SecondVirialKGRearranged}
    \begin{aligned}
    &
    \frac{\lambda'(t)}{2\lambda(t)}\intR \varphi'\left(\frac{x+\mu(t)}{\lambda(t)}\right)
    \left(\frac{x+\mu(t)}{\lambda(t)}\right)\bigg(|u_x|^2
    +|u|^2 + |u_t|^2 + \frac12 |v|^2 + \frac12 |n|^2 + n|u|^2  \bigg)(t,x) dx\\
    &- \frac{\mu'(t)}{2\lambda(t)}\intR \varphi'\left(\frac{x+\mu(t)}{\lambda(t)}\right)\bigg(|u_x|^2
    +|u|^2 + |u_t|^2 + \frac12 |v|^2 + \frac12 |n|^2 + n|u|^2 \bigg)(t,x) dx\\
    & =-\deri J(t) +\frac{1}{\lambda(t)}\intR \varphi'\left(\frac{x+\mu(t)}{\lambda(t)}\right)
    \bigg(\frac{1}{2} vn+ \frac12 v|u|^2 - u_t u_{x}
    \bigg)(t,x) dx.
    \end{aligned}
  \end{equation}
We have that, by Gagliardo-Nirenberg inequality and Lemma \ref{PropositionFiniteNormsKGZ}
  \[
    \begin{aligned}
    &\intR \varphi'\left(\frac{x+\mu(t)}{\lambda(t)}\right)
    \bigg(\frac{1}{2} vn+ \frac12 v|u|^2 - u_t u_{x}
    \bigg)(t,x) dx \\
    &\lesssim \|v(t)\|^2_{L^2(\R)}+ \|n(t)\|^2_{L^2(\R)}+\|u_x(t)\|^2_{H^1(\R)}+\|u_t(t)\|_{L^2(\R)}^2+\|u(t)\|_{L^2(\R)}^6 \le K,
    \end{aligned}
  \]
where $K>0$, is a constant depending on the energy and the $L^2$-norm of $u$.
Thus, we integrate equation \eqref{SecondVirialKGRearranged} in time over $[2, \infty)$ and get that
  \begin{equation*}
    \begin{aligned}
    &\int_2^{\infty}
    +\frac{\lambda'(t)}{2\lambda(t)}\intR \varphi'\left(\frac{x+\mu(t)}{\lambda(t)}\right)
    \left(\frac{x+\mu(t)}{\lambda(t)}\right)\bigg(|u_x|^2
    +|u|^2 + |u_t|^2 + \frac12 |v|^2 + \frac12 |n|^2 + n|u|^2     \bigg)(t,x) dx\\
    &- \frac{\mu'(t)}{2\lambda(t)}\intR \varphi'\left(\frac{x+\mu(t)}{\lambda(t)}\right)\bigg(|u_x|^2
    +|u|^2 + |u_t|^2 + \frac12 |v|^2 + \frac12 |n|^2 + n|u|^2 \bigg)(t,x) dx dt <\infty.
    \end{aligned}
  \end{equation*}
Note that by Sobolev embbedings and \eqref{hip1theorem4},
  \begin{equation}\label{final}
    \begin{aligned}
      -\intR \phi\left(\frac{x+\mu(t)}{\lambda(t)}\right)
      \left(|n||u|^2\right)(t,x)dx & \ge
      -\intR \phi\left(\frac{x+\mu(t)}{\lambda(t)}\right)
      \left(\frac{1}{3}|n(t,x)|^2+\frac34|u(t,x)|^4\right)dx\\
      &\ge-\intR \phi\left(\frac{x+\mu(t)}{\lambda(t)}\right)
      \left(\frac{1}{3}|n(t,x)|^2+\frac34|u(t,x)|^2\right)dx.
    \end{aligned}
  \end{equation}
Then, we argue as in Subsection \ref{SecondPartOfTheProof} and obtain that there exists a sequence of time $\{t_n\}$, $t_n \to \infty$, such that,
  \begin{equation}\label{TermsTendingToZeroEnergyKG}
    \begin{aligned}
    \intR \left|\varphi'\left(\frac{x+\mu(t_n)}{\lambda(t_n)}\right)\right|
    \bigg(|u_x|^2
    +|u|^2 + |u_t|^2 + \frac12 |v|^2 + \frac12 |n|^2 +n|u|^2\bigg)(t_n,x) dx
     \to 0.
    \end{aligned}
  \end{equation}
  As in Section \ref{DecayinFarRegions}, we consider $\phi \in C_0^1(\R)$
  such that $\phi(s) \in [0,1] \text{ for all } s\in \R$,
  $\text{supp}(\phi)=[-3/4, -1,4]$,
  \[
  \phi(s)\lesssim|\varphi'(s)|\text{ and } |\phi'(s)|\lesssim|\varphi'(s)|
  \text{ for all }s \in \R.
  \]
  Following the computations for $\varphi$, one gets,
  	\begin{equation*}
  	\begin{aligned}
  	& \left|\dfrac{d}{dt}\intR
    \phi\left(\frac{x+\mu(t)}{\lambda(t)}\right)\bigg(|u_x|^2+|u|^2+|u_t|^2
    +\frac12 |n|^2+\frac12|v|^2+n|u|^2\bigg)(t,x)dx \right|\\
  	&\quad \lesssim  \frac{1}{\lambda(t)}
    \intR\left| \phi'\left(\frac{x+\mu(t)}{\lambda(t)}\right)\right|
    \left|\bigg( \frac{1}{2} vn+ \frac12 v|u|^2 -
    u_t u_{x}\bigg)(t,x)\right| dx\\
  	&\quad {}~ +\frac{\mu'(t)}{2\lambda(t)}\intR
    \left|\phi'\left(\frac{x+\mu(t)}{\lambda(t)}\right)\right|
    \left|\Big(|u_x|^2+|u|^2+|u_t|^2+\frac12|v|^2+\frac12|n|^2
    +n|u|^2\Big)(t,x)\right| dx.
  	\end{aligned}
  	\end{equation*}

  Integrate over $[t, t_n]$, $t \ge T$
  %
  and take $t_n\to \infty$. Thanks to \eqref{TermsTendingToZeroEnergyKG}, we obtain
    \begin{equation*}
      \begin{aligned}
        & \left|\intR \phi\left(\frac{x+\mu(t)}{\lambda(t)}\right)\bigg(|u_x|^2+ |u|^2+|u_t|^2+\frac12 |n|^2+\frac12|v|^2+n|u|^2\bigg)(t,x)dx \right|\\
        &\lesssim \int_t^{\infty} \frac{1}{\lambda(\tau)} d\tau +\int_t^{\infty} \frac{\mu'(\tau)}{2\lambda(\tau)}\intR \left|\phi'\left(\frac{x+\mu(\tau)}{\lambda(\tau)}\right)\right|\left|\Big(|u_x|^2 + |u|^2+ |u_t|^2+|n|^2+|v|^2+n|u|^2\Big)(\tau,x)\right| dx d\tau
      \end{aligned}
    \end{equation*}
  Now, taking $t\to \infty$, we obtain
    \begin{equation*}
      \begin{aligned}
        \lim_{t\to \infty} \left|\intR \phi\left(\frac{x+\mu(t)}{\lambda(t)}\right)
        \bigg(|u_x|^2+ |u|^2+|u_t|^2
        +\frac12 |n|^2+\frac12|v|^2+n|u|^2\bigg)(t,x)dx \right|\le 0
      \end{aligned}
    \end{equation*}
Consequently, taking into account \eqref{final}, one gets
  \begin{equation*}
    \begin{aligned}
      &\lim_{t\to \infty} \intR \phi\left(\frac{x+\mu(t)}{\lambda(t)}\right)\bigg(|u_x|^2+\frac14 |u|^2+|u_t|^2+\frac16 |n|^2+\frac12|v|^2\bigg)(t,x)dx\\
      &\le\lim_{t\to \infty} \intR \phi\left(\frac{x+\mu(t)}{\lambda(t)}\right)\bigg(|u_x|^2+ |u|^2+|u_t|^2
      +\frac12 |n|^2+\frac12|v|^2+n|u|^2\bigg)(t,x)dx
      =0.
    \end{aligned}
  \end{equation*}
Then, \eqref{tesis1theorem4} follows.

\end{document}